\documentclass[12pt]{amsart}
\usepackage{fullpage,url,amssymb,enumerate,colonequals}
\usepackage[all]{xy} % for complicated commutative diagrams, use pdf driver
\usepackage{mathrsfs} % for \mathscr (script letters)
\usepackage[OT2,T1]{fontenc}
\DeclareSymbolFont{cyrletters}{OT2}{wncyr}{m}{n}
\DeclareMathSymbol{\Sha}{\mathalpha}{cyrletters}{"58}

\usepackage{color}

\newcommand{\dashedarrow}{\dashrightarrow}

\newcommand{\Aff}{\mathbb{A}}
\newcommand{\C}{\mathbb{C}}
\newcommand{\F}{\mathbb{F}}

\newcommand{\Gmk}{\mathbb{G}_{\mathrm{m},k}}

\newcommand{\PP}{\mathbb{P}}
\newcommand{\Q}{\mathbb{Q}}

\newcommand{\Z}{\mathbb{Z}}

\newcommand{\uom}{\underline{\omega}}

\newcommand{\calC}{\mathcal{C}}

\newcommand{\calF}{\mathcal{F}}

\newcommand{\calJ}{\mathcal{J}}

\newcommand{\calO}{\mathcal{O}}

\DeclareMathOperator{\id}{id}

\DeclareMathOperator{\Log}{Log}

\DeclareMathOperator{\rank}{rank}

\newcommand{\smooth}{{\operatorname{smooth}}}

\newcommand{\unr}{{\operatorname{unr}}}

\newcommand{\isom}{\simeq}

\newcommand{\rholog}{\rho \log}

\newcommand{\To}{\longrightarrow}

\newcommand{\pws}[1]{[\![#1]\!]}

\numberwithin{equation}{section}
\newtheorem{theorem}{Theorem}[section]
\newtheorem{lemma}[theorem]{Lemma}
\newtheorem{corollary}[theorem]{Corollary}
\newtheorem{proposition}[theorem]{Proposition}

\theoremstyle{definition}
\newtheorem{definition}[theorem]{Definition}
\newtheorem{question}[theorem]{Question}
\newtheorem{conjecture}[theorem]{Conjecture}

\theoremstyle{remark}
\newtheorem{remark}[theorem]{Remark}

\definecolor{darkgreen}{rgb}{0,0.5,0}
\usepackage[
        colorlinks, citecolor=darkgreen,
        backref,
        pdfauthor={Michael Stoll}, % add other authors
        pdftitle={Uniform bounds for the number of rational points %
                  on hyperelliptic curves of small Mordell-Weil rank},
]{hyperref}
\usepackage[alphabetic,backrefs,lite]{amsrefs} % for bibliography

\usepackage{comment}

\begin{document}

\title[Uniform bounds for rational points]%
      {Uniform bounds for the number of rational points\\
       on hyperelliptic curves of small Mordell-Weil rank}
%\subjclass[2010]{Primary 11G30; Secondary 14G25, 14G40, 14K15, 14K20}

\author{Michael Stoll}
\address{Mathematisches Institut,
         Universit\"at Bayreuth,
         95440 Bayreuth, Germany.}
\email{Michael.Stoll@uni-bayreuth.de}
\urladdr{http://www.computeralgebra.uni-bayreuth.de}

\date{November 25, 2015}

\begin{abstract}
  We show that there is a bound depending only on~$g$, $r$ and~$[K:\Q]$
  for the number of $K$-rational points on a hyperelliptic curve~$C$ of genus~$g$
  over a number field~$K$ such that the Mordell-Weil rank~$r$
  of its Jacobian is at most~$g-3$. If $K = \Q$, an explicit bound is
  $8rg + 33(g - 1) + 1$.

  The proof is based on Chabauty's method; the new ingredient is an estimate
  for the number of zeros of an abelian logarithm on a $p$-adic `annulus'
  on the curve, which generalizes the standard bound on disks.
  The key observation is that for a $p$-adic field~$k$, the set of
  $k$-points on~$C$ can be covered by a collection of disks and annuli
  whose number is bounded in terms of~$g$ (and~$k$).

  We also show, strengthening a recent result by Poonen and the author,
  that the lower density of hyperelliptic curves of odd degree over~$\Q$
  whose only rational point is the point at infinity tends to~$1$ uniformly
  over families defined by congruence conditions, as the genus~$g$ tends to infinity.
\end{abstract}

\maketitle

%%%%%%%%%%%%%%%%%%%%%%%%%%%%%%%%%%%%%%%%%%%%%%%%%%%%%%%%%%%%%%%%%%%%%%%%%%%

\section{Introduction}

Since Faltings' proof~\cite{Faltings1983} of Mordell's conjecture, we know
that a curve of \hbox{genus~$g \ge 2$} over a number field~$K$
can have only finitely many $K$-rational points.
This raises the question whether there might be uniform bounds of some
sort on the number of $K$-rational points.
Caporaso, Harris and Mazur~\cite{CaporasoHarrisMazur1997} have
shown that the validity of the weak Lang
conjecture on rational points on varieties of general type would imply
the existence of a bound depending only on the genus~$g$ and the field~$K$.
Pacelli~\cite{Pacelli1997} has, under the same assumption, shown the existence
of a bound depending only on~$g$ and the degree of~$K$.
(For function fields like $k = \F_p(t)$, the number of $k$-points on curves
over~$k$ of fixed genus is unbounded, however, see for
example~\cite{ConceicaoUlmerVoloch}.)
On the other hand,
considering an embedding of the curve into its Jacobian variety, which
identifies the set of $K$-rational points on the curve with the intersection
of the curve and the Mordell-Weil group, one can formulate the following purely
geometric statement (Mazur~\cite{Mazur1986}*{end of Section~III.2}
asks it as a question):

\begin{conjecture}[Uniform Mordell-Lang for curves] \label{Conj:uML}
  Given $g \ge 2$ and $r \ge 0$, there is a constant $N(g, r)$ such that
  for any curve $C$ over~$\C$ of genus~$g$ with an embedding $i \colon C \to J$
  into its Jacobian and for any subgroup $\Gamma \subset J(\C)$ of rank~$r$, one has
  $\#i^{-1}(\Gamma) \le N(g, r)$.
\end{conjecture}

That this number is finite for each individual curve and subgroup
follows from further work by Faltings~\cite{Faltings1994}.
Heuristic arguments suggest that such a uniform bound should exist.
The existence of such bounds has been shown for $k$ a \emph{function
field} of characteristic~zero if $C$ is not defined over the algebraic numbers
by Buium~\cite{Buium1993} (and also for function fields in characteristic~$p$ by
Buium and Voloch~\cite{BuiumVoloch}). In Section~\ref{S:ZBuML} below,
we will show that Conjecture~\ref{Conj:uML} is implied by (a special case of)
the Zilber-Pink conjecture; this implication can be seen as making precise the
`heuristic arguments' alluded to above.

A weaker variant of Conjecture~\ref{Conj:uML}, turning the geometric statement
into an arithmetic one, is the following.

\begin{conjecture} \label{Conj:dgr}
  Given $d \ge 1$, $g \ge 2$ and $r \ge 0$, there is a constant $R(d,g,r)$
  such that for any number field $K$ of degree~$d$ and any curve $C$ over~$K$
  of genus~$g$ with Jacobian~$J$ such that $\rank J(K) = r$, we have
  $\#C(K) \le R(d,g,r)$.
\end{conjecture}

This is formulated as a question again by Mazur in~\cite{Mazur2000}*{page~223} (allowing
the constant to depend on~$K$, not just on the degree~$d$).

However, to our knowledge, so far
not even a uniform (and unconditional) bound for the number of
\emph{rational torsion points} on curves of some fixed genus~$g \ge 2$
has been obtained! In this note, we finally obtain such a bound for
\emph{hyperelliptic} curves of genus at least~$3$.
More generally, we can show that on a hyperelliptic curve~$C$ of genus~$g$
over a number field of degree~$\le d$, there can be at most
$R(d,g,r)$ rational points mapping
into a given subgroup of rank $r \le g-3$ of the Mordell-Weil group,
where $R(d,g,r)$ depends only on~$d$, $g$ and~$r$. This
implies uniform bounds in terms of~$d$, $g$ and~$r$
for the number of rational points on such curves
as long as the Mordell-Weil rank is at most~$g-3$,
and also for the number of rational points in a torsion packet when $g \ge 3$,
see Theorem~\ref{T:main-ratpoints-general} and Corollary~\ref{C:torsion}
below. In particular, this proves Conjecture~\ref{Conj:dgr} for hyperelliptic
curves when $r \le g-3$.

The proof is based on Chabauty's
method~\cites{Chabauty1941,Coleman1985chabauty,McCallum-Poonen2013,Stoll2006-chabauty},
whose `classical' version we now sketch.
If $C$ is a curve over~$\Q$, with Jacobian~$J$ and minimal regular model
$\calC$ over~$\Z_p$,
where the prime $p$ is sufficiently large and we assume
that $r = \rank J(\Q) < g$, then one can bound $\#C(\Q)$ by the number of
smooth $\F_p$-points on the special fiber of~$\calC$ plus~$2r$,
see~\cite{Katz-Zureick-Brown}. This bound is obtained as follows.
Consider the Chabauty-Coleman pairing (defined below in Section~\ref{S:notation})
\[ \Omega^1_J(\Q_p) \times J(\Q_p) \To \Q_p, \quad
   (\omega, P) \longmapsto \oint_O^P \omega
\]
This pairing is $\Q_p$-linear
in~$\omega$ and additive in~$P$; its kernel on the left is trivial. If $r < g$, then
there is a linear subspace $V \subset \Omega^1_J(\Q_p)$ of dimension at
least~$g-r \ge 1$ that annihilates the Mordell-Weil group $J(\Q) \subset J(\Q_p)$
under the pairing. Let $P_0 \in C(\Q)$ and use $P_0$ as basepoint
for an embedding $i \colon C \to J$. Then for all $P \in C(\Q)$ and
all $\omega \in V$, we have
\[ 0 = \oint_O^{i(P)} \omega = \oint_{P_0}^P i^* \omega \]
where $i^* \omega \in \Omega^1_C(\Q_p)$ is a regular differential on~$C$.
The integral on the right is defined by this equality.
One then shows (see for example~\cite{Stoll2006-chabauty}) that the number
of zeros of the function
\[ P \longmapsto \oint_{P_0}^P i^* \omega \]
on a $p$-adic residue disk of~$C$, which is the set of $p$-adic points reducing
mod~$p$ to a given smooth point on the special fiber of~$\calC$, is at most
one plus the number of zeros (counted with multiplicity) of~$\omega$ on
that residue disk. (Here we use that $p$ is large enough, otherwise the
bound has to be modified.) Choosing a `good' $\omega \in V$ for each residue disk
leads to the bound
\[ \#C(\Q) \le \#\calC(\F_p)^\smooth + 2r \]
mentioned earlier.

The problem with this approach is that the bound depends on the complexity
of the special fiber of~$\calC$, which is unbounded --- there can be
arbitrarily long chains of rational curves in the special fiber, which
can lead to an arbitrarily large number of smooth $\F_p$-points.
The idea for overcoming this problem is to parameterize the subset of~$C(\Q_p)$
corresponding to such a chain not by a union of (an unbounded number of) disks,
but by an `annulus'.
Such an annulus arises as the set of $p$-adic points on~$C$ reducing to
an ordinary double point on the special fiber of a suitable (not necessarily
regular) model of the curve, which is obtained by contracting the chain.
We can then obtain a bound for the number of
points in that subset that is independent of the number of residue disks
it contains.
Since both the number of such annuli and the number of remaining residue
disks are bounded in terms of the genus (and~$p$),
see Theorem~\ref{T:specialfiber}, we do obtain a uniform bound.
The price we have to pay is that on (at least some of) the annuli, we
need to impose additional linear conditions on the differential~$\omega$,
so that we need the space of differentials annihilating the relevant
subgroup of~$J(\Q_p)$ to be of dimension at least three. This translates
into the rank bound $r \le g-3$. The key result for our application
is Proposition~\ref{P:integration-A}, which gives a precise comparison
of the abelian integral pulled back to an annulus and the $p$-adic
integral of the pulled-back $1$-form. It turns out that the difference
between the two is a linear function of the valuation.

We carry out this approach in the case of hyperelliptic curves.
Our method does in fact generalize to arbitrary curves as demonstrated
by recent work of Katz, Rabinoff and Zureick-Brown~\cite{KRZB}.

For the convenience of the reader, we give an overview of the proof
of the main result, which we state here in simplified form.

\begin{theorem}[Theorem~\ref{T:local-main}] \label{T:main}
  Let $k$ be a $p$-adic field with $p$ odd and write $e$ for the ramification
  index of~$k$ and $q$ for the size of its residue field. Let $g \ge 3$ and $0 \le r \le g-3$.
  We assume that $p > e+1$.

  Let $C \colon y^2 = f(x)$ be a hyperelliptic curve of genus~$g$ over~$k$.
  We denote by~$J$ the Jacobian variety of~$C$. Let $\Gamma \subset J(k)$ be
  a subgroup of rank~$r$.
  Let $i \colon C \to J$ be an embedding given by choosing some basepoint~$P_0 \in C(k)$.
  Then
  \[ \#\{P \in C(k) : i(P) \in \Gamma\} = O\bigl((e(r+1) + q)g\bigr) \,. \]
\end{theorem}

Applying (a precise version of) this result for $k = \Q_3$ to a curve
over~$\Q$ and to $\Gamma = J(\Q)$ leads
to the following bound for the number of rational points.

\begin{theorem}[Theorem~\ref{T:main-ratpoints-general} for $d=1$]
  Let $g \ge 3$ and $0 \le r \le g-3$. Let $C$ be a hyperelliptic curve of genus~$g$
  over~$\Q$ such that the Mordell-Weil rank of its Jacobian is~$r$. Then
  \[ \#C(\Q) \le 33(g - 1) + 1 \quad \text{if $r = 0$} \qquad \text{and} \qquad
     \#C(\Q) \le 8rg + 33(g - 1) - 1 \quad \text{if $r \ge 1$.}
  \]
\end{theorem}

The proof of Theorem~\ref{T:main} proceeds in the following steps.
\begin{enumerate}[1.]\addtolength{\itemsep}{1mm}
  \item We show that $C(k)$ can be partitioned into $O(qg)$~disks and $O(g)$~annuli
        (Proposition~\ref{P:DAbounds}).
  \item On the union of the disks, $\#i^{-1}(\Gamma)$ can be bounded by~$O(qg + er)$
        by the usual Chabauty method (Lemma~\ref{L:points-D}).
  \item We give a bound of the form~$O(e(r+1))$ for $\#i^{-1}(\Gamma)$ on an annulus
        (Proposition~\ref{P:good-A}). This is where we need the stronger condition
        $r \le g-3$ compared to the usual Chabauty condition $r \le g-1$.
        As already mentioned,
        the reason behind this is that we want to use a different integral
        that satisfies the Fundamental Theorem of Calculus on annuli to get the bound.
        We therefore have to compare this integral with the abelian integral
        used in the Chabauty-Coleman pairing. The result is that both agree
        when the differential satisfies two extra linear conditions
        (Proposition~\ref{P:integration-A}).
  \item Finally, we add the bounds for the disks and the annuli.
\end{enumerate}

The paper is organized as follows. In Section~\ref{S:ZBuML}, we show
that a version of the Zilber-Pink Conjecture implies Conjecture~\ref{Conj:uML}.
This is not essential for the main results of the paper, but gives some idea
regarding the kind of bound in terms of $g$ and~$r$ one might expect to hold.
After a short section introducing notation, we proceed in Section~\ref{S:AG}
with a discussion of the combinatorics of reduction graphs. If one is only
interested in the existence of some bound (as long as $r \le g-3$), then
it suffices to use
the result of Artin and Winters~\cite{ArtinWinters} that gives the existence
of bounds in terms of~$g$ for the number of chains and `$\Aff^1$-components'.
The precise results given by Theorem~\ref{T:specialfiber} are only needed to
obtain the concrete bounds in the statements of Theorems \ref{T:local-main}
and~\ref{T:main-ratpoints-general}. Section~\ref{S:part} uses the main
result of Section~\ref{S:AG} to give bounds in terms of~$g$ for the number
of disks and annuli needed to cover the set of $p$-adic points on~$C$.
Section~\ref{S:hyp} gives an explicit description of the
annuli on a hyperelliptic curve when the residue characteristic is odd.
Section~\ref{S:integ} compares the abelian integral on an annulus with the
integral satisfying the Fundamental Theorem of Calculus and deduces
a bound on
the number of common zeros on an annulus of abelian integrals coming from
differentials killing a given subgroup of~$J$ under the Chabauty-Coleman pairing.
The next two Sections \ref{S:mainT} and~\ref{S:ratpts} then combine the
results of Sections \ref{S:part} and~\ref{S:integ} to state and prove our main result,
Theorem~\ref{T:local-main}, and its application to bounds for rational points,
Theorem~\ref{T:main-ratpoints-general}, and for rational torsion packets,
Corollary~\ref{C:torsion}. The last section, Section~\ref{S:rholog},
uses the generalization of our results on differentials on annuli
due to Katz, Rabinoff and Zureick-Brown
to deduce a version of the main result of~\cite{PoonenStoll2014} that
applies uniformly to families of odd degree hyperelliptic curves that are
defined by congruence conditions.

\subsection*{Acknowledgments}

The vague idea that one should be able to use Chabauty's method to prove
uniform upper bounds for the number of rational points had long been in
the author's mind, but was put aside as infeasible because of the
apparent problems described above. The new activity leading to the
results presented here was prompted by a question Manjul Bhargava
asked related to~\cite{PoonenStoll2014}: could we give a
family of odd degree hyperelliptic curves~$C$ of arbitrarily high genus, defined by
congruences, such that our method would not work for any curve
in the family? The intuition that this should not be possible for
large genus led to the idea of using integration on annuli to prove
that the size of the image of $C(\Q_2)$ in~$\PP^{g-1}(\F_2)$ under the `$\rholog$'
map of~\cite{PoonenStoll2014} is bounded by a polynomial in~$g$.
This result (with a quadratic bound) is given in Section~\ref{S:rholog} below.
The idea then extended naturally to the original problem.
So I would like to thank Manjul for asking the right question.
I also wish to thank Amnon Besser for help with questions
about $p$-adic integration
and Stefan Wewers for answering my questions on stable models (which have
by now been eliminated from the argument, but see Remark~\ref{R:variants}).
Dino Lorenzini was very helpful on the question (discussed in
Section~\ref{S:AG}) of how to bound
the number of `$\Aff^1$-components' in the special fiber of the
minimal regular model of a curve. Felipe Voloch provided some pointers
to the literature.
The idea for proving that Zilber-Pink implies uniform Mordell-Lang
for curves germinated upon hearing a talk by Umberto Zannier
at the joint \"OMG and DMV meeting in Innsbruck in September~2013 and
took shape while reading his book~\cite{ZannierBook} afterwards.
Padmavathi Srinivasan asked some questions that helped improve the
argument in the proof of Theorem~\ref{T:specialfiber}.
Last, but not least, I would like to thank an anonymous referee
for spotting a mistake and for making some valuable suggestions that
led to improvements in organizing the arguments in Sections \ref{S:AG}
and~\ref{S:hyp}.

%%%%%%%%%%%%%%%%%%%%%%%%%%%%%%%%%%%%%%%%%%%%%%%%%%%%%%%%%%%%%%%%%%%%%%%%%%%

\section{Zilber-Pink implies uniform Mordell-Lang for curves} \label{S:ZBuML}

In~\cite{Pink2005preprint}*{Conjecture~6.1}, Pink formulates
a more general version of the following conjecture. It is a special case of
a conjecture on mixed Shimura varieties that belongs to a circle of ideas
usually referred to as the `Zilber-Pink conjecture(s)'.

\begin{conjecture}[Pink] \label{ConjPink}
  Let $\pi \colon A \to B$ be an algebraic family of abelian varieties over~$\C$.
  Consider an irreducible subvariety $X \subset A$ of dimension~$d$
  such that $X$ is not contained in any proper closed subgroup scheme of~$A$.
  Then the set of points $x \in X$ that are contained in a subgroup
  of codimension $> d$ of the fiber~$A_{\pi(x)}$ above $\pi(x) \in B$
  is not Zariski dense in~$X$.
\end{conjecture}

The idea behind this is that based on the dimensions, one would not expect
any intersection between $X$ and a subgroup scheme of codimension $> d$,
so intersection points are `unlikely' and should therefore form a `sparse'
subset of~$X$. See Zannier's book~\cite{ZannierBook} for background information
on the subject of `unlikely intersections'.

(Pink's original version is for families of semi-abelian varieties. However,
Bertrand~\cites{Bertrand2011preprint,Bertrand2013} gave a counterexample to this more general
formulation. It turns out that the semi-abelian version needs to be modified
to be compatible with the original conjecture on mixed Shimura varieties.)

In this section we show that Conjecture~\ref{ConjPink} implies Conjecture~\ref{Conj:uML}.
The strategy is similar to that employed by Caporaso, Harris, and
Mazur in~\cite{CaporasoHarrisMazur1997}.
Namely, we show that Pink's conjecture implies that if a curve has many
points whose differences generate a subgroup of bounded rank in the Jacobian,
then the points have algebraic dependencies, similar to what is implied by
`correlation' in the sense of~\cite{CaporasoHarrisMazur1997} under the
weak Lang conjecture.
In more or less the same way as in that paper, the result then follows.

Let $\pi \colon \calC \to B$ be a smooth family of irreducible curves
of genus~$g$ over~$\C$, with $B$ (say, irreducible) of dimension~$d$.
We write $\calJ \to B$ for the induced family of
Jacobians. Fix $r \ge 0$. Given $n > r$, consider the $n$-th fiber power
$\calC_B^n \to B$. We denote by $\phi$ the morphism
\[ \calC_B^n \To \calJ_B^{n-1}, \qquad
   (b; P_0, P_1, \ldots, P_{n-1}) \longmapsto (b; [P_1-P_0], \ldots, [P_{n-1}-P_0]) \,.
\]
We claim that the image of~$\phi$ is not contained in a proper subgroup
scheme of~$\calJ_B^{n-1}$. Consider a point $b \in B$ and fix a basepoint
$P_0 \in \calC_b$. Since the image of~$\calC_b$ in~$\calJ_b$ under the embedding
$P \mapsto [P - P_0]$ spans~$\calJ_b$ as a group, it follows that the image
of $\calC_b^n$ in~$\calJ_b^{n-1}$ spans the latter group. In particular,
this image cannot be contained in a proper algebraic subgroup of~$\calJ_b^{n-1}$.
Since a proper subgroup scheme of~$\calJ_B^{n-1}$ will meet most fibers
in a proper subgroup of the fiber, this shows that $\phi(\calC_B^n)$ cannot
be contained in a proper subgroup scheme of~$\calJ_B^{n-1}$.

If the subgroup of the Jacobian generated by the point differences has
rank at most~$r$, then there are $n-1-r$ independent relations of the form
\[ a_{i1}[P_1-P_0] + a_{i2}[P_2-P_0] + \ldots + a_{i,n-1}[P_{n-1}-P_0] = 0 \]
with integers~$a_{ij}$. For points $(b; Q_1, \ldots, Q_{n-1}) \in \calJ_B^{n-1}$,
the relations
\[ a_{i1} Q_1 + a_{i2} Q_2 + \ldots + a_{i,n-1} Q_{n-1} = 0 \]
then define a subgroup scheme of~$\calJ_B^{n-1}$
containing $\phi(b; P_0, P_1, \ldots, P_{n-1})$
and of codimension $(n-1-r)g$.
The dimension of the image of~$\phi$ is at most $\dim \calC_B^n = d + n$.
So the codimension is greater than this dimension whenever
\begin{equation} \label{E:lb}
  n > \frac{d + g}{g-1} + \frac{g}{g-1} r \,.
\end{equation}
We conclude:

\begin{lemma} \label{L:start}
  Assume Conjecture~\ref{ConjPink}. If $d$, $g$, $r$ and $n$ satisfy~\eqref{E:lb},
  then the set of points in~$\calC_B^n$ such that the differences lie in
  a subgroup of rank $\le r$ is not Zariski dense.
\end{lemma}

Now we mimic the argument given in~\cite{CaporasoHarrisMazur1997}*{Section~1.2}.
We first prove the following lemma.

\begin{lemma} \label{L:ind}
  Assume Conjecture~\ref{ConjPink}. Let $\pi \colon \calC \to B$ be a smooth
  family of irreducible curves of genus~$g \ge 2$ over~$\C$. Fix $r \ge 0$.
  Then there is a bound $N(\pi, r)$ and a proper closed subvariety~$B'$ of~$B$
  such that for all $b \in B(\C) \setminus B'(\C)$, and for any choice of
  strictly more than~$N(\pi, r)$
  distinct points on the curve~$\calC_b$, the differences of these points
  will generate a subgroup of rank strictly greater than~$r$ in the Jacobian~$\calJ_b$.
\end{lemma}

\begin{proof}
  Fix some $n$ satisfying~\eqref{E:lb} for the given values of $g$, $r$,
  and $d = \dim B$.
  Denote by $Z_n \subset \calC_B^n$ the Zariski closure of the set of
  points $(b;P_0,P_1,\ldots,P_{n-1}) \in \calC_B^n$ such that the differences
  of the~$P_j$ generate a subgroup of rank $\le r$. By Lemma~\ref{L:start},
  $Z_n$ is a proper closed subvariety of~$\calC_B^n$. Now for $1 \le j \le n$,
  we let $\rho_j \colon \calC_B^j \to \calC_B^{j-1}$ denote the forgetful
  morphism that leaves out the last point. For $j = n-1, n-2, \ldots, 0$,
  define successively $Z_j$ as the (closed) subvariety of~$\calC_B^j$ of
  points~$x$ such that $\rho_{j+1}^{-1}(x) \subset Z_{j+1}$. Since (inductively)
  $Z_{j+1}$ is a proper closed subvariety of~$\calC_B^{j+1}$, $Z_j$
  is a proper closed subvariety of~$\calC_B^j$. We let $B' = Z_0 \subset \calC_B^0 = B$.

  Arguing as in~\cite{CaporasoHarrisMazur1997}*{Proof of Lemma~1.1},
  there are integers~$d_j$ such that $\#\rho_j^{-1}(x) \cap Z_j \le d_j$
  for all $x \in \calC_B^{j-1} \setminus Z_{j-1}$.
  We now show by downward induction the following statement.

  {\em Let $0 \le m \le n$. Then there is $N_m \ge m$ such that
  for each $(b;P_0,P_1,\ldots,P_{m-1}) \in \calC_B^m \setminus Z_m$,
  whenever we choose $N_m-m+1$ distinct additional
  points $P_m, P_{m+1}, \ldots, P_{N_m} \in \calC_b$, then
  the differences of the~$P_j$ generate a subgroup of rank $> r$ in~$\calJ_b$.}

  For $m = n$ we can take $N_n = n$, by definition of~$Z_n$. Now let $m < n$
  and assume the claim is true for $m+1$ in place of~$m$. Let
  $x = (b;P_0,P_1,\ldots,P_{m-1}) \in \calC_B^m \setminus Z_m$,
  then there are at most~$d_{m+1}$ points in~$Z_{m+1}$ mapping to~$x$.
  By the inductive assumption, if we choose points $P_m, \ldots, P_{N_{m+1}}$
  with $P_m$ not one of the finitely many possibilities leading to a preimage
  in~$Z_{m+1}$, then the statement is true. In any case, once we take
  more than~$d_{m+1}$ additional (distinct) points, then at least one of
  them will lead to a preimage outside~$Z_{m+1}$. Since we can permute the
  additional points, this brings us back to the previous case. We see that
  we can take $N_m = \max\{N_{m+1}, m+d_{m+1}\}$.

  The final case $m = 0$ then gives the statement of the lemma, with
  \[ N(\pi, r) = N_0 = \max\{m + d_{m+1} : 0 \le m \le n\} \]
  (where $d_{n+1} \colonequals 0$).
\end{proof}

Now we are almost done.

\begin{theorem}
  Conjecture~\ref{ConjPink} implies Conjecture~\ref{Conj:uML}.
\end{theorem}

\begin{proof}
  Assume Conjecture~\ref{ConjPink}.
  Fix $g \ge 2$ and~$r \ge 0$ and let $\calC_0 \to B_0$ be a universal family
  of smooth curves of genus~$g$. By Lemma~\ref{L:ind}, there is a proper closed
  subvariety $B_1 \subset B_0$ and a bound~$N_0$ such that the statement
  of Conjecture~\ref{Conj:uML} holds with this bound for all fibers of~$\calC_0$
  above points not in~$B_1$. If $B_1 \neq \emptyset$, we can apply Lemma~\ref{L:ind}
  to the restricted family $\calC_1 \to B_1$ and obtain a proper closed subvariety
  $B_2 \subset B_1$ and a bound~$N_1$ valid for all fibers above points outside~$B_2$.
  We continue this process, which must stop after finitely many steps since $B$
  is noetherian. The statement of Conjecture~\ref{Conj:uML} then holds
  with $N(g,r) = \max_j N_j$.
\end{proof}

\begin{remark}
  The same argument shows that there is such a uniform bound for any smooth family
  of curves inside abelian varieties of dimension at least~$2$ that are
  generated fiber-wise by the curves.
\end{remark}

Note that we can take $\dim B_j \le \dim B_0 = 3g - 3$. Looking at~\eqref{E:lb},
this implies that it suffices to take $n = 5 + 2r$. So we would expect that
except for points occurring systematically in certain families of curves,
there should be a bound of the
form $\ll r + 1$ for the number of points on a curve mapping into a subgroup
of rank~$r$ in the Jacobian. For hyperelliptic curves of genus~$g$, taking a Weierstrass point
as basepoint, we always have the $2g+2$ Weierstrass points mapping to points
of order~$2$ (and no other systematically occurring torsion points,
see~\cite{PoonenStoll2014}*{Section~7}).
Since any generically chosen additional set of $r$~pairs of `opposite' points
on such a curve will generate a subgroup of rank~$r$, we obtain a lower
bound of $2g + 2 + 2r \gg g + r$. In~\cite{Stoll2006-chabauty} we show that
for the family of quadratic twists of a fixed hyperelliptic curve (and over
any fixed number field~$K$), there is an upper bound of $2g + 2 + 2r$ for the
number of $K$-rational points whenever $r < g$, with at most finitely many exceptions.
In this paper, we prove an upper bound $\ll_{[K:\Q]} (r + 1) g$ for the set of $K$-rational
points when the curve is hyperelliptic and $r \le g-3$.
It appears possible that the method can be refined to give a bound of
the form $\ll_{[K:\Q]} g + r$. This leads to the following question.

\begin{question}
  Can we take $R(d,g,r) \ll_d g + r$ in Conjecture~\ref{Conj:dgr}?
  Can we perhaps even take $N(g,r) \ll g + r$ in Conjecture~\ref{Conj:uML}?
\end{question}

%%%%%%%%%%%%%%%%%%%%%%%%%%%%%%%%%%%%%%%%%%%%%%%%%%%%%%%%%%%%%%%%%%%%%%%%%%%

\section{Notation} \label{S:notation}

Until further notice, we fix the following notation.

Let $p$ be a prime number. As usual, $\Q_p$ denotes the field of $p$-adic numbers
and $\C_p$ the completion of an algebraic closure of~$\Q_p$. We let
$v \colon \C_p \to \Q \cup \{\infty\}$ denote the additive valuation on~$\C_p$,
normalized by $v(p) = 1$. We also fix the absolute value $|x| = p^{-v(x)}$
on~$\C_p$. Throughout the paper, $k \subset \C_p$ stands for a finite field extension of~$\Q_p$
with ramification index~$e$; we write $\calO$ for its
ring of integers and $\kappa$ for the
residue field. We set $q \colonequals \#\kappa$;
$k^\unr \subset \C_p$ is the maximal unramified extension of~$k$.

Let $g \ge 3$ be an integer and let $C$ be a smooth, projective and
geometrically integral curve of genus~$g$ over~$k$.
The Jacobian variety of~$C$ is denoted~$J$; the origin on~$J$ is~$O$.
We denote the image of the divisor $(P)-(Q)$ on~$C$ in~$J$ by $[P-Q]$.
We denote by~$\log_J$ the $p$-adic abelian logarithm map $J(k) \to T_O J(k) \cong k^g$.
On a sufficiently small subgroup neighborhood of~$O$, it is given by evaluating
the formal logarithm, and then extended to all of~$J(k)$ by linearity.
The space~$\Omega^1_J(k)$ of global regular $1$-forms on~$J$ defined over~$k$
agrees with the
space of invariant (under translations) $1$-forms on~$J$ and can be identified
with the cotangent space $(T_O J(k))^*$ of~$J$ at the origin. This induces a pairing
\[ \Omega^1_J(k) \times J(k) \To k, \quad
    (\omega, P) \longmapsto \langle \omega, \log_J(P) \rangle
                        \mathrel{=:} \oint_O^P \omega\,,
\]
which we call the \emph{Chabauty-Coleman pairing}. It is $k$-linear in~$\omega$
and additive (and $\calO$-linear on the kernel of reduction) in~$P$.
Its kernel on the left is trivial, and its kernel on
the right is the torsion subgroup of~$J(k)$.

Let $P_0 \in C(k)$ and let $i \colon C \to J$ be the embedding given by
$P \mapsto [P-P_0]$. Then $i^* \colon \Omega^1_J \to \Omega^1_C$ is an
isomorphism (which does not depend on~$P_0$). If $\omega \in \Omega^1_C(k)$
is $i^* \omega_J$ for some $\omega_J \in \Omega^1_J(k)$,
then we set for points $P, Q \in C(k)$
\[ \oint_P^Q \omega \colonequals \oint_{i(P)}^{i(Q)} \omega_J
                    = \oint_O^{[Q-P]} \omega_J \,.
\]
We use the symbol $\oint$ to distinguish this integral defined via abelian
logarithms from the integral~$\int$ given by $p$-adic integration
theory. This distinction will be relevant in Section~\ref{S:integ}.

Inclusions `$A \subset B$' are meant to be non-strict.

%%%%%%%%%%%%%%%%%%%%%%%%%%%%%%%%%%%%%%%%%%%%%%%%%%%%%%%%%%%%%%%%%%%%%%%%%%%

\section{Combinatorics of arithmetic graphs} \label{S:AG}

In this section, we study the combinatorics of the (smooth part of the)
special fiber of the minimal regular model~$\calC$ over~$\calO$ of a (smooth
projective geometrically integral) curve~$C$ of genus~$g \ge 2$ over~$k$.
For the general background, we refer to~\cite{Liu-book}*{Sections 9 and~10.1}.

The special fiber~$\calC_s$ of~$\calC$ decomposes
into irreducible components; we assume for now that the residue field~$\kappa$
is large enough so that the components are geometrically irreducible.
Let $\Gamma$ be one of these components of~$\calC_s$.
If $W$ denotes a relative canonical divisor, then by the adjunction formula we have,
writing as usual $p_a(\Gamma)$ for the arithmetic genus of~$\Gamma$,
\begin{equation} \label{E:adjunction}
  \Gamma \cdot W = 2 p_a(\Gamma) - 2 - \Gamma^2 \,.
\end{equation}
By~\cite{Liu-book}*{Corollary~9.3.26}, $g \ge 2$ implies that $\Gamma \cdot W \ge 0$.
So there are two cases: $\Gamma \cdot W > 0$ and $\Gamma \cdot W = 0$.
If $m(\Gamma)$ denotes the multiplicity of~$\Gamma$ in~$\calC_s$, then
\begin{equation} \label{E:adjglobal}
  2g-2 = \calC_s \cdot W = \sum_\Gamma m(\Gamma) (\Gamma \cdot W) \,,
\end{equation}
which implies that there can be at most $2g-2$ components~$\Gamma$
having $\Gamma \cdot W > 0$, with components counted according to multiplicity.
On the other hand, $\Gamma \cdot W = 0$
means $p_a(\Gamma) = 0$ and $\Gamma^2 = -2$ or $p_a(\Gamma) = 1$ and
$\Gamma^2 = 0$ (the intersection pairing is negative semidefinite,
so $\Gamma^2 \le 0$). $\Gamma^2 = 0$ would imply that $\Gamma$ is the
only component; then $2g-2 = 0$ and so $g = 1$, which we have excluded.
So $\Gamma$ is isomorphic to~$\PP^1$ over~$\kappa$ and has self-intersection~$-2$.
Such components are called \emph{$(-2)$-curves}.

Associated to the special fiber~$\calC_s$ is a graph~$G$, whose vertices
correspond to the components of~$\calC_s$, with two (distinct) vertices
$\Gamma_1$ and~$\Gamma_2$ joined by $\Gamma_1 \cdot \Gamma_2$ edges.
The graph~$G$ is connected. To each vertex~$\Gamma$ we associate
its multiplicity~$m(\Gamma)$ and its arithmetic genus~$p_a(\Gamma)$.
We call $G$ the \emph{arithmetic graph} associated to~$\calC$.
This data is equivalent to what is called a `type' in~\cite{ArtinWinters}
or~\cite{Liu-book}*{Definition~10.1.55}.
The intersection pairing satisfies
\[ \Gamma \cdot \sum_{\Gamma'} m(\Gamma') \Gamma' = \Gamma \cdot \calC_s = 0 \,. \]
Using the adjunction formula~\eqref{E:adjunction}, we can write this as
\begin{equation} \label{E:GammaW}
  \sum_{\Gamma' \neq \Gamma} m(\Gamma') \Gamma \cdot \Gamma'
     = -m(\Gamma) \Gamma^2
     = m(\Gamma) (\Gamma \cdot W + 2) - 2 m(\Gamma) p_a(\Gamma) \,.
\end{equation}

We are interested in the structure of the smooth part~$\calC_s^\smooth$
of the special fiber. It is the union of the components of multiplicity~$1$
minus their singular points and the points where they meet other components.
We have already seen that there can be at most~$2g-2$ components~$\Gamma$
of multiplicity~$1$ and with $\Gamma \cdot W > 0$.
The remaining components of~$\calC_s^\smooth$ are $(-2)$-curves
of multiplicity~$1$, so by~\eqref{E:GammaW} the total intersection
number with other components is~$2$. We note that not all components
of~$\calC_s$ can be $(-2)$-curves, since then \hbox{$2g-2 = W \cdot \calC_s$} would vanish,
contradicting the assumption $g \ge 2$. This implies that there cannot be
three $(-2)$-curves of multiplicity~$1$ meeting in one point or two meeting
in one point with intersection multiplicity~$2$, since in these cases there
could be no other components.
There are therefore the following possibilities for
how a $(-2)$-curve~$\Gamma$ of multiplicity~$1$ can meet other components.
\begin{enumerate}[(1)]\addtolength{\itemsep}{1mm}
  \item $\Gamma$ meets two components of multiplicity~$1$ in two distinct points.
        Then $\Gamma$ is part of a maximal \emph{chain} of such components that
        connects two components of multiplicity~$1$ (which can be identical)
        that are not $(-2)$-curves.
  \item \label{case2} $\Gamma$ meets a component of multiplicity~$2$ in one point.
  \item \label{case3}
        $\Gamma$ meets two components $\Gamma'$, $\Gamma''$ of multiplicity~$1$
        in the same point such that
        \begin{enumerate}[(\ref{case3}a)]
          \item \label{case3a} either none of $\Gamma'$, $\Gamma''$ is a $(-2)$-curve, or
          \item \label{case3b} $\Gamma'$ is a $(-2)$-curve, but $\Gamma''$ is not.
        \end{enumerate}
  \item \label{case4}
        $\Gamma$ meets a component of multiplicity~$1$, which is not a $(-2)$-curve,
        in one point with intersection multiplicity~$2$.
\end{enumerate}
In cases \eqref{case2} to~\eqref{case4}, $\Gamma \cap \calC_s^\smooth$
is isomorphic to~$\Aff^1$. We will call such components of~$\calC_{s}$
simply \emph{$\Aff^1$-components}.

In general, there can also be chains consisting of $(-2)$-curves of higher
(constant) multiplicity. They do not form part of~$\calC_s^\smooth$,
so they are not of interest for our purposes.
Artin and Winters~\cite{ArtinWinters}*{Theorem~1.6} show that there
are only finitely many different `types' of fixed genus up to an equivalence
that ignores the lengths of chains of any multiplicity as above.
This implies in particular that there must be bounds that depend only
on~$g$ for the number of (maximal) chains of $(-2)$-curves of multiplicity~$1$
and for the number of $\Aff^1$-components.
The following result gives explicit and optimal such bounds.

\begin{theorem} \label{T:specialfiber}
  Let $\calC_s$ be the special fiber of the minimal proper regular model
  of a smooth projective geometrically integral curve~$C$ of genus~$g \ge 2$
  over a $p$-adic field~$k$. Then there are numbers $t, u \ge 0$ with
  $t + u \le g - a$, where $a$ denotes the abelian rank of the special fiber
  of the N\'eron model of the Jacobian of~$C$, such that
  \begin{enumerate}[\upshape (i)]\addtolength{\itemsep}{1mm}
    \item \label{spf1} The number of components $\Gamma$ of~$\calC_s$
                       with $\Gamma \cdot W > 0$ is $N \le 2g-2$.
    \item \label{spf2} The number of maximal chains of $(-2)$-curves of multiplicity~$1$
                       in~$\calC_s$ is at most \\ $N-1+t \le 2g-3+t$.
    \item \label{spf3} The number of $\Aff^1$-components in~$\calC_s$ is at most~$3u$.
  \end{enumerate}
\end{theorem}

\begin{remark}
  It is not very hard to construct an arithmetic graph of genus~$g$
  with $2g-2$ components~$\Gamma$ such that $\Gamma \cdot W > 0$ and having
  $2g-3+t$ chains and $3(g-t)$ $\Aff^1$-components, for every $t = 0,1,\ldots,g$.
  We leave this as an exercise for the interested reader. This shows that
  the bounds given in the theorem above are optimal.
\end{remark}

\begin{proof}
  We have $N \le 2g-2$ by~\eqref{E:adjglobal}.

  We note that in terms of the graph~$G$ associated to the special
  fiber~$\calC_s$, a component as in case \eqref{case3a} or~\eqref{case4} above
  is indistinguishable from a chain of length~$1$, and the two $\Aff^1$-components
  involved in case~\eqref{case3b} are indistinguishable from a chain of length~$2$.
  (Indeed, after a slight deformation of the special fiber that does not change
  the intersection multiplicities of the components, the point of intersection
  breaks up into two or three ordinary double points, and the respective components
  do form a chain of length $1$ or~$2$.) Write $c$ for the number of maximal chains,
  $d_{\text{3a}}$, $d_{\text{3b}}$, $d_{\text{4}}$ for the number of $\Aff^1$-components
  as in cases \eqref{case3a}, \eqref{case3b} and~\eqref{case4} above, and $d$
  for the number of remaining $\Aff^1$-components. We show that there are numbers
  $t', u' \ge 0$ with $t' + u' \le g - a$ such that
  \begin{equation} \label{claim}
    c + d_{\text{3a}} + \frac{1}{2} d_{\text{3b}} + d_{\text{4}} \le N - 1 + t'
     \qquad\text{and}\qquad
     d \le 3 u' \,.
  \end{equation}
  Claims \eqref{spf2} and~\eqref{spf3} follow by taking $t = t' - \delta$
  and $u = u' + \delta$ with
  $\delta = d_{\text{3a}} + \frac{1}{2} d_{\text{3b}} + d_{\text{4}}$
  (note that $d_{\text{3b}}$ is even).

  We write $\chi(G) = 1 - t'$ for the Euler characteristic of~$G$, where
  $t'$ denotes the number of independent loops in~$G$.

  We now bound the number of chains together with the `false chains' coming
  from $\Aff^1$-components in cases \eqref{case3a}, \eqref{case3b} or~\eqref{case4},
  Consider the subgraph~$G'$ of~$G$
  spanned by the $N$~vertices corresponding to components~$\Gamma$
  with $\Gamma \cdot W > 0$ and by the vertices corresponding
  to components in chains (false or otherwise).
  Contracting each of these chains to an edge, we obtain
  a graph~$G''$ whose Euler characteristic equals that of~$G'$, which in turn
  cannot be smaller than that of~$G$ (since $G$ is connected). So we find that
  \begin{align*}
    c + d_{\text{3a}} + \frac{1}{2} d_{\text{3b}} + d_{\text{4}}
     &= \#\{\text{chains}\} + \#\{\text{false chains}\} \\
     &\le \#\{\text{edges of $G''$}\}
      = N - \chi(G'')
      \le N - \chi(G)
      = N - 1 + t'
  \end{align*}
  as claimed in the first inequality in~\eqref{claim}.

  To obtain a bound on the number~$d$ of the remaining $\Aff^1$-components,
  we classify the vertices $\Gamma$ of~$G$ according to the pair
  $(m(\Gamma), \Gamma \cdot W) \in \Z_{>0} \times \Z_{\ge 0}$ of invariants.
  Given $m \ge 1$ and $w \ge 0$, we call a vertex~$\Gamma$ of~$G$
  with $m(\Gamma) = m$ and $\Gamma \cdot W = w$ an \emph{$(m,w)$-vertex}.
  We denote by $v_{(m,w)}$ the number of $(m,w)$-vertices.
  We consider each edge of~$G$ as an oriented edge with both
  possible choices of orientation. We then denote by~$e_{(m,w),(m',w')}$
  the number of oriented edges leading from an $(m,w)$-vertex to
  an $(m',w')$-vertex. We also write $p_{(m,w)}$ for the sum of~$p_a(\Gamma)$
  over all $(m,w)$-vertices~$\Gamma$.

  Taking the sum of~\eqref{E:GammaW} over all $(m,w)$-vertices, we obtain
  \[  m (w + 2) v_{(m,w)} - 2 m p_{(m,w)} = \sum_{(m',w')} m' e_{(m,w),(m',w')}\,, \]
  or equivalently,
  \begin{equation} \label{E:relation1}
    v_{(m,w)} = \frac{1}{m(w+2)} \sum_{(m',w')} m' e_{(m,w),(m',w')} + \frac{2}{w+2} p_{(m,w)}\,,
  \end{equation}
  which allows us to replace $v_{(m,w)}$ by the right hand side.
  If we use this in~\eqref{E:adjglobal}, this gives
  \begin{equation} \label{E:relation2}
    2g - 2
      = \sum_{(m,w)} m w \,v_{(m,w)}
      = \sum_{(m,w),(m',w')} \frac{w m'}{w+2} e_{(m,w),(m',w')}
          + \sum_{(m,w)} \frac{2 w m}{w+2} p_{(m,w)} \,.
  \end{equation}
  In addition, remembering that $G$ is connected and that for $(m,w) \neq (m',w')$,
  the sum
  $e_{(m,w),(m',w')} + e_{(m',w'),(m,w)}$ counts twice the number of edges between
  vertices with invariants $(m,w)$ and~$(m',w')$, whereas $e_{(m,w),(m,w)}$ counts twice
  the number of edges between $(m,w)$-vertices, we have the relation
  \[ 2 \sum_{(m,w)} v_{(m,w)} - 2 + 2 t' = \sum_{(m,w),(m',w')} e_{(m,w),(m',w')} \,, \]
  which we rewrite using~\eqref{E:relation1} as
  \[ \sum_{(m,w),(m',w')} \Bigl(\frac{2m'}{m(w+2)} - 1\Bigr) e_{(m,w),(m',w')}
        + \sum_{(m,w)} \frac{4}{w+2} p_{(m,w)}
      = 2 - 2 t' \,.
  \]
  Adding \eqref{E:relation2} to this, we finally have
  \begin{align} \label{E:relation3}
    \sum_{(m,w),(m',w')} \Bigl(\frac{m'(mw+2)}{m(w+2)} - 1\Bigr) e_{(m,w),(m',w')}
       &= 2(g - t' - p') - 2 \sum_{(m,w)} \frac{w(m-1)}{w+2} p_{(m,w)} \\
       &\le 2(g - t' - p')\,, \nonumber
  \end{align}
  where we have set
  \[ p' = \sum_{(m,w)} p_{(m,w)} = \sum_\Gamma p_a(\Gamma) \ge \sum_\Gamma p_g(\Gamma) = a \,. \]
  Here $p_g(\Gamma)$ denotes the geometric genus of~$\Gamma$.
  We set $u' = g - t'- p'$; then $t' + u' = g - p' \le g - a$.
  Let `$<$' denote the lexicographical ordering of the pairs~$(m,w)$.
  Using the obvious equality $e_{(m,w),(m',w')} = e_{(m',w'),(m,w)}$,
  we can rewrite~\eqref{E:relation3} as
  \begin{align}
    \sum_{(m,w)} \frac{(m-1)w}{w+2} e_{(m,w),(m,w)}
      {} + \!\!\!\!\sum_{(m,w) < (m',w')}\!\!
          \Bigl(\frac{m'(mw+2)}{m(w+2)} + \frac{m(m'w'+2)}{m'(w'+2)} - 2\Bigr)
                              e_{(m,w),(m',w')}
      \le 2u' \,.  \label{E:relation4}
  \end{align}
  We can bound the coefficient of~$e_{(m,w),(m',w')}$ in~\eqref{E:relation4} from below:
  \begin{align*}
    \frac{m'(mw+2)}{m(w+2)} + \frac{m(m'w'+2)}{m'(w'+2)} - 2
      &= m' - 2 \frac{m'(m-1)}{m(w+2)} + m - 2 \frac{m(m'-1)}{m'(w'+2)} - 2 \\
      &\stackrel{\text{\makebox[0mm][c]{$w \ge 0$}}}{\ge} m' - \frac{m'(m-1)}{m} + m - \frac{m(m'-1)}{m'} - 2 \\
      &= \frac{m'}{m} + \frac{m}{m'} - 2 \ge 0 \,.
  \end{align*}
  So all coefficients on the left hand side of~\eqref{E:relation4} are nonnegative;
  the coefficient of~$e_{(m,w),(m,w)}$ vanishes if and only if $w = 0$ or $m = 1$, and
  the coefficient of~$e_{(m,w),(m',w')}$ vanishes if and only if we have equality
  everywhere in the above, which is equivalent to $m = m' = 1$ (or $m = m'$ and $w = w' =0$,
  but then $(m,w) = (m',w')$).

  Let $\lambda_{(m,w),(m',w')}$ denote the coefficient of $e_{(m,w),(m',w')}$
  in~\eqref{E:relation4}. Then
  \begin{align*}
    \lambda_{(1,w),(2,0)} &= \frac{1}{2} \quad \text{for all $w \ge 0$}, %\\
    &\lambda_{(1,0),(2,w')} &\ge \frac{2}{3} \quad \text{for all $w' \ge 1$}, \\
    \lambda_{(2,0),(2,w')} &\ge \frac{1}{3} \quad \text{for all $w' \ge 1$}, %\\
    &\lambda_{(2,0),(3,w')} &\ge \frac{1}{6} \quad \text{for all $w' \ge 0$}.
  \end{align*}
  Using this in~\eqref{E:relation4} we obtain
  \begin{equation} \label{E:relation5}
    \frac{1}{2} e_{(1,0),(2,0)} + \frac{1}{2} \sum_{w \ge 1} e_{(1,w),(2,0)}
      + \frac{1}{3} \sum_{w' \ge 1} e_{(2,0),(2,w')}
      + \frac{1}{6} \sum_{w' \ge 0} e_{(2,0),(3,w')}
      + \frac{2}{3} \sum_{w' \ge 1} e_{(1,0),(2,w')}
      \le 2 u' \,.
  \end{equation}
  We now claim that
  \begin{equation} \label{E:crucial}
    3 \sum_{w \ge 1} e_{(1,w),(2,0)}
      + 2 \sum_{w' \ge 1} e_{(2,0),(2,w')}
      + \sum_{w' \ge 0} e_{(2,0),(3,w')} \ge e_{(1,0),(2,0)} \,.
  \end{equation}
  Assuming this for a moment, we can use~\eqref{E:crucial} in~\eqref{E:relation5}
  to obtain
  \[ \frac{2}{3} \sum_{w' \ge 0} e_{(1,0),(2,w')} \le 2 u'
    \qquad\text{or equivalently,}\qquad
    \sum_{w' \ge 0} e_{(1,0),(2,w')} \le 3 u' \,.
  \]
  The left hand side counts exactly the number~$d$ of $(-2)$-curves of multiplicity~$1$
  that meet a component of multiplicity~$2$, so this finishes the proof
  of claim~\eqref{claim}.

  It remains to prove~\eqref{E:crucial}. We first observe that contracting
  an edge between two $(2,0)$-vertices does not change
  the genus or the topological properties of~$G$ and also does not
  affect~\eqref{E:relation4}. So we can assume without loss of generality
  that no such edges are present. Equivalently, we can consider chains of
  $(2,0)$-vertices instead of single $(2,0)$-vertices in the argument below.
  We now consider those $(2,0)$-vertices
  that contribute to $e_{(1,0),(2,0)}$, i.e., that have an edge
  to a $(1,0)$-vertex. Let $a_j$ ($1 \le j \le 3$) denote the number of such
  vertices~$\Gamma$ such that the highest multiplicity of a vertex connected
  to~$\Gamma$ is~$j$. Since $g \ge 2$, there cannot be a $(2,0)$-vertex
  connected only to $(1,0)$-vertices, as this would give rise to a connected
  component of genus~$1$, contradicting the fact that $G$ is connected.
  This implies that a vertex counted by~$a_j$ can have at most $(4-j)$ edges
  to $(1,0)$-vertices; it also has at least one edge to a vertex with multiplicity~$j$
  that is not a $(1,0)$-vertex. So
  \[ \sum_{w \ge 1} e_{(1,w),(2,0)} \ge a_1, \quad
     \sum_{w' \ge 1} e_{(2,0),(2,w')} \ge a_2, \quad
     \sum_{w' \ge 0} e_{(2,0),(3,w')} \ge a_3,
  \]
  and therefore
  \[ e_{(1,0),(2,0)}
      \le 3 a_1 + 2 a_2 + a_3
      \le 3 \sum_{w \ge 1} e_{(1,w),(2,0)}
          + 2 \sum_{w' \ge 1} e_{(2,0),(2,w')}
          + \sum_{w' \ge 0} e_{(2,0),(3,w')}
  \]
  as claimed.
\end{proof}

\begin{remark}
  One can in fact take $t$ and~$u$ in Theorem~\ref{T:specialfiber} to be the
  toric and unipotent ranks
  of the special fiber of the N\'eron model of the Jacobian of~$C$.
  For claim~\eqref{spf2}, this follows from a similar argument as in the
  proof above, but using the bipartite graph~$G'$ whose vertices correspond to
  the components of the special fiber and the intersection points of components,
  with edges whenever a point lies on a component. This version of the
  reduction graph avoids the `false' chains and satisfies $1 - \chi(G') \le $
  the toric rank~$t$, compare~\cite{Liu-book}*{Exercise~10.1.19}.

  For claim~\eqref{spf3}, we recall from the proof above that the bound on the number
  of $\Aff^1$-components
  is $3(g - t' + \delta - p')$. We have $t' - \delta + p' \ge t + a = g - u$,
  where $u$ denotes the unipotent rank, so that $g - t' + \delta - p' \le u$.
  Note that $t' - \delta$ is still an upper bound for the part of the toric rank
  coming from loops in the configuration of components --- `false' chains give
  rise to `false' loops --- whereas $p'$ is an upper bound
  for $a$ plus the part of the toric rank coming from individual components. \\[1ex]
  For our intended application, the version as given in Theorem~\ref{T:specialfiber}
  is sufficient, though.
\end{remark}

In general, some of the components of~$\calC_s$ may not be defined
over~$\kappa$. If a chain contains a component defined over~$\kappa$,
then either all components of the chain are defined over~$\kappa$,
or else the chain contains an odd number of components of which only
the middle one is defined over~$\kappa$ (and the action of Frobenius
reverses the orientation of the chain).

%%%%%%%%%%%%%%%%%%%%%%%%%%%%%%%%%%%%%%%%%%%%%%%%%%%%%%%%%%%%%%%%%%%%%%%%%%%

\section{Partition into disks and annuli} \label{S:part}

We keep the notation introduced so far. Let $P \in C(k)$ be a point.
Then $P$ reduces to a point $\bar{P} \in \calC_s^\smooth(\kappa)$,
and so $\bar{P}$ is either on a component~$\Gamma$ with $\Gamma \cdot W > 0$
(and multiplicity~$1$), or on an $\Aff^1$-component, or on a component
belonging to a chain. We bound the number of smooth $\kappa$-points
occurring in the first two cases. Let $a$, $t$ and~$u$ be as in
Theorem~\ref{T:specialfiber}; we can assume that $a + t + u = g$.
Denoting by~$p_g(\Gamma)$ the geometric
genus of the component~$\Gamma$ and writing $\Gamma_1, \ldots, \Gamma_{N'}$
for the components occurring in the first case (with $N' \le N$, since we
only consider components defined over~$\kappa$ and with multiplicity~$1$),
we obtain the bound
\[ \sum_{j=1}^{N'} \bigl(q + 1 + 2 p_g(\Gamma_j) \sqrt{q}\bigr)
     \le (2g-2)(q+1) + 2 \sum_{j=1}^{N'} p_g(\Gamma_j) \sqrt{q}
     \le (2g-2)(q+1) + 2a \sqrt{q}
\]
for the number of smooth $\kappa$-points on components having positive
intersection with~$W$. For the number of smooth $\kappa$-points
on $\Aff^1$-components, we have the bound $3 u q$, since each \hbox{$\Aff^1$-component}
defined over~$\kappa$ has $q$ smooth $\kappa$-points. For $a+u = g-t$ fixed,
the sum of these bounds is maximal when $a = 0$, leading to a bound of
\[ (2g-2)(q+1) + 3(g-t)q = (5q+2) (g-1) - 3 q (t-1) \]
for the number of smooth $\kappa$-points outside components belonging to chains.
Each such point~$P$ gives rise to a \emph{residue disk}, which is the subset
of $C(k)$ of points reducing to~$P$; these subsets are analytically isomorphic
to the sets of $k$-points of open $p$-adic disks over~$k$ in the following sense.

\begin{definition} \label{D:disk}
  We let $D_{0,k}$ denote the $p$-adic analytic \emph{open unit disk} over~$k$.
  Its ring of \emph{analytic regular functions} is the subring of~$k\pws{z}$
  of power series converging whenever $|z| < 1$ (for a power series
  $f(z) = \sum_{n=0}^\infty a_n z^n$, this means that $|a_n| r^n \to 0$
  for all $0 < r < 1$). We call any analytic isomorphism $u \colon D_{0,k} \to D_{0,k}$
  a \emph{coordinate} on~$D_{0,k}$. It can be checked that for any analytic
  map $h \colon D_{0,k} \to D_{0,k}$, the map~$u$ given by $u(z) = z(1 + h(z))$
  is a coordinate on~$D_{0,k}$.

  For $k \subset K \subset \C_p$ a field extension, we set
  $D_0(K) = \{\xi \in K : |\xi| < 1\}$.

  An \emph{(open) disk in~$C$} is an injective analytic map $\varphi \colon D_{0,k} \to C$
  (i.e., given by coordinates that are analytic regular functions on~$D_{0,k}$).
\end{definition}

Now consider a maximal chain of $(-2)$-curves of multiplicity~$1$
in the special fiber~$\calC_s$. Its two ends each meet some other component of
multiplicity~$1$ transversally. Contracting the components in the chain,
we obtain another model~$\calC'$ of~$C$ such that the image of the chain
in~$\calC'_s$ is an ordinary double point~$Q$. (We consider only chains containing
a component defined over~$\kappa$. If the action of Frobenius reverses the
orientation of the chain, we replace $k$ by its unramified extension of
degree~$2$, so that the Frobenius action is trivial. Since the bound we will
obtain for the number of relevant points in the residue annulus of~$Q$
does not depend on~$q$ and so is valid even for $k^\unr$-points, we do not
lose anything in this way.) By~\cite{BoschLuetkebohmert-stabredI}*{Proposition~2.3},
the preimage of~$Q$ in~$C(k)$ under the reduction map is analytically isomorphic
to the $k$-points of an open annulus over~$k$ in the sense of Definition~\ref{D:annulus}
below. The number of such annuli equals the number of chains (defined over~$\kappa$)
and so is bounded according to Theorem~\ref{T:specialfiber} by~$2g - 3 + t$.

\begin{definition} \label{D:annulus}
  Let $0 < \alpha < 1$ be such that $\alpha = |\xi|$ for some $\xi \in k$.
  We let $A_{\alpha,k}$ denote the standard $p$-adic \emph{open annulus} over~$k$
  of height~$\alpha$. Its ring of \emph{analytic regular functions} is the ring of
  (infinite in both directions) Laurent series in~$z$ converging whenever
  $\alpha < |z|  < 1$ (for $f(z) = \sum_{n=-\infty}^\infty a_n z^n$, this means
  that $\lim_{n \to \pm\infty} |a_n| r^n = 0$ for all $\alpha < r < 1$).
  We call any analytic isomorphism $u \colon A_{\alpha,k} \to A_{\alpha,k}$
  a \emph{coordinate} on~$A_{\alpha,k}$. It can be checked that for any analytic
  map $h \colon A_{\alpha,k} \to D_{0,k}$, the map~$u$ given by $u(z) = z(1 + h(z))$
  is a coordinate on~$A_{\alpha,k}$
  (see for example~\cite{BGR-book}*{Lemma~9.7.1/1 and Proposition~9.7.1/2},
  applied to all closed annuli in~$A_{\alpha,K}$ for $k \subset K \subset \C_p$).

  For $k \subset K \subset \C_p$ a field extension, we set
  $A_\alpha(K) = \{\xi \in K : \alpha < |\xi| < 1\}$.

  An \emph{(open) annulus in~$C$} is an injective analytic map $\varphi \colon A_{\alpha,k} \to C$
  (i.e., given by coordinates that are analytic regular functions on~$A_{\alpha,k}$),
  for some $\alpha$ as above.
\end{definition}

Summarizing the discussion above, we have shown:
\begin{proposition} \label{P:DAbounds}
  Let $C$ be a smooth projective geometrically integral curve over~$k$ of genus~$g$.
  Then there is a number $0 \le t \le g$ such that $C(k)$ can be written as a
  disjoint union of the sets of $k$-points of at most $(5q+2) (g-1) - 3 q (t-1)$
  open disks and at most $2(g-1) + (t-1)$ open annuli in~$C$.
\end{proposition}

Let $C_D(k)$ be the union of the disks and $C_A(k)$ the union of the annuli
in this partition.

%%%%%%%%%%%%%%%%%%%%%%%%%%%%%%%%%%%%%%%%%%%%%%%%%%%%%%%%%%%%%%%%%%%%%%%%%%%

\section{Annuli in hyperelliptic curves} \label{S:hyp}

In this section we give an explicit description of the annuli on
a hyperelliptic curve. This is used in Section~\ref{S:integ} below to obtain
bounds for the number of points on an annulus that map into a given subgroup
of the Jacobian. We do this for a $p$-adic field~$k$ when $p$ is \emph{odd}.
We proceed in three steps, as follows.
\begin{enumerate}[1.]\addtolength{\itemsep}{1mm}
  \item We construct disks and annuli in~$C$ from disks and annuli in~$\PP^1$.
  \item We give a classification of analytic involutions on disks and annuli.
  \item We use Step~2 to show that all annuli in~$C$ arise as in Step~1.
\end{enumerate}
We then use this explicit description to describe the restriction of the global
regular differentials on~$C$ to these annuli.

We begin with the construction of disks and annuli in the hyperelliptic
curve~$C$. We write $\iota \colon C \to C$ for the hyperelliptic involution,
$\pi \colon C \to \PP^1$ for the hyperelliptic double cover
and $\Theta \subset \PP^1$ for its set of branch points; note that $\#\Theta = 2g+2$
is even. We can assume that $\infty \notin \Theta$; then an equation for~$C$ is given by
\[ y^2 = f(x) = c \prod_{\theta \in \Theta} (x - \theta) \]
with some $c \in k^\times$. For $0 \neq \theta \in \Theta$ we set
\[ f_\theta^+(x) = \Bigl(1 - \frac{\theta}{x}\Bigr)^{1/2} \in k\pws{x^{-1}}
   \qquad\text{and}\qquad
   f_\theta^-(x) = \Bigl(1 - \frac{x}{\theta}\Bigr)^{1/2} \in k\pws{x} \,.
\]
Note that $f_\theta^+(x)$ converges for $|x| > |\theta|$ and that
$f_\theta^-(x)$ converges for $|x| < |\theta|$ (here we use that $p$ is odd).
We then have
\[ x - \theta = x f_\theta^+(x)^2 \qquad\text{and}\qquad
   x - \theta = -\theta f_\theta^-(x)^2
\]
for $x$ in the respective domain of convergence.

\begin{lemma} \label{L:disk}
  Let $\varphi \colon D_{0,k} \stackrel{\isom}{\to} D \subset \PP^1_k$ be an open disk.
  \begin{enumerate}[\upshape(1)]\addtolength{\itemsep}{1mm}
    \item \label{Ldisk1}
          If $D(\C_p) \cap \Theta = \emptyset$, then let $\xi \in D(k)$.
          If $f(\xi)$ is not a square in~$k$, then $\pi^{-1}(D) \cap C(k)$ is empty.
          Otherwise $\pi^{-1}(D)$ is the union of
          two disjoint open disks in~$C$, each isomorphic to~$D$ via~$\pi$.
    \item \label{Ldisk2}
          If $D(\C_p) \cap \Theta = \{\theta_1\}$, then $\theta_1 \in k$.
          We assume that the radius~$r$ of~$D$ (in terms of a coordinate on~$\PP^1$
          such that $\infty \notin D$) satisfies $r |f'(\theta_1)| = |\xi|^2$
          for some $\xi \in k$. Then $D' = \pi^{-1}(D)$ is a disk in~$C$.
          In terms of suitable coordinates on $D'$ and~$D$,
          $\pi \colon D' \to D$ is given by $z \mapsto z^2$.
          The hyperelliptic involution acts on~$D'$ as $z \mapsto -z$.
    \item \label{Ldisk3}
          If $D(\C_p) \cap \Theta = \{\theta_1, \theta_2\}$ has two elements,
          then $(x-\theta_1)(x-\theta_2)$ has coefficients in~$k$. The set
          $\pi^{-1}(D) \cap C(k)$ is either contained in
          the preimage of the smallest closed disk containing $\theta_1$ and~$\theta_2$,
          or else $\pi^{-1}(D)$ is an annulus~$A$ in~$C$ such that
          in terms of suitable coordinates on $A$ and~$D$, $\pi \colon A \to D$
          is given by $z \mapsto z + \beta/z$ with some $\beta \in k^\times$.
          The hyperelliptic involution acts on~$A$ as $z \mapsto \beta/z$.
  \end{enumerate}
\end{lemma}

\begin{proof}
  We can (after possibly a coordinate change on~$\PP^1$)
  assume that $\varphi = \id$ and $D = D_0$. Then in case~\eqref{Ldisk1},
  we can take $\xi = 0$, and we have $|\theta| \ge 1$ for all $\theta \in \Theta$.
  So on~$D$ we can write the equation of~$C$ as
  \[ y^2 = c \prod_{\theta \in \Theta} \theta \cdot
             \Bigl(\prod_{\theta \in \Theta} f_\theta^-(x)\Bigr)^2
         = c' h(x)^2 \,,
  \]
  where $c' = c \prod_{\theta \in \Theta} \theta = f(0)$ and
  $h(x) = \prod_{\theta \in \Theta} f_\theta^-(x)$.
  If $c'$ is not a square in~$k$, then this equation has
  no solution in~$k$ and so $\pi^{-1}(D)$ does not contain $k$-points of~$C$.
  Otherwise write $c' = \gamma^2$ for some $\gamma \in k^\times$.
  Then $\pi^{-1}(D)$ is the disjoint union of
  \[ D^+ = \bigl\{(\xi, \gamma h(\xi)) : \xi \in D\bigr\}
     \qquad\text{and}\qquad
     D^- = \bigl\{(\xi, -\gamma h(\xi)) : \xi \in D\bigr\} \,,
  \]
  and the projection to the first coordinate $\pi \colon D^\pm \to D$ is an
  analytic isomorphism.

  In case~\eqref{Ldisk2}, we first observe that $\theta_1$ must be fixed under the
  action of the absolute Galois group of~$k$, since $D$ and~$\Theta$ are; it follows
  that $\theta_1 \in k$. We can then in addition assume that $\theta_1 = 0$.
  Since we assume that $D = D_0$, we have $r = 1$.
  Similarly as in case~\eqref{Ldisk1} we write the equation of~$C$ on~$D$ as
  \[ y^2 = -c \prod_{0 \neq \theta} \theta \cdot x
           \Bigl(\prod_{0 \neq \theta \in \Theta} f_\theta^-(x)\Bigr)^2
         = c' x h(x)^2 \,,
  \]
  where $c' = -c \prod_{0 \neq \theta \in \Theta} \theta = f'(\theta_1)$ and
  $h(x) = \prod_{0 \neq \theta \in \Theta} f_\theta^-(x)$.
  Choosing $\gamma \in k$ and $u \in \calO^\times$ such that $\gamma^2 = u c'$,
  we can now parameterize $D_{0,k} \stackrel{\isom}{\to} D' = \pi^{-1}(D)$ via
  \[ z \longmapsto \bigl(u z^2, \gamma z h(u z^2)\bigr) \,. \]
  If we use $u^{-1} x$ as the coordinate on~$D$,
  then $\pi \colon D' \to D$ is given by $z \mapsto z^2$. It is clear that
  the hyperelliptic involution is given in terms of~$z$ by $z \mapsto -z$.
  We remark that the condition `$r |f'(\theta_1)| = |\xi|^2$ for some $\xi \in k$'
  is invariant under coordinate transformations.

  In case~\eqref{Ldisk3}, we see in the same way as before that the set $\{\theta_1, \theta_2\}$
  is fixed by the action of the absolute Galois group of~$k$, which implies that
  the coefficients of $(x-\theta_1)(x-\theta_2)$ are in~$k$. We can then change
  coordinates so that $\theta_1 + \theta_2 = 0$ (and $D$ is still the open unit disk).
  Let $\theta_1 \theta_2 = a \in k^\times$ ($\theta_1$ and~$\theta_2$ must be nonzero,
  since $f$ does not have multiple roots).
  Set $\Theta' = \Theta \setminus \{\theta_1, \theta_2\}$; then the equation of~$C$
  on~$D$ can be written as
  \[ y^2 = c \prod_{\theta \in \Theta'} \theta \cdot (x^2 - a)
           \Bigl(\prod_{\theta \in \Theta'} f_\theta^-(x)\Bigr)^2
         = c' (x^2 - a) h(x)^2 \,,
  \]
  where $c' = c \prod_{\theta \in \Theta'} \theta$ and
  $h(x) = \prod_{\theta \in \Theta'} f_\theta^-(x)$.
  If $c'$ is not a square in~$k$, then there are no solutions when $|x| > |\theta_1| = |\theta_2|$,
  and $\pi^{-1}(D) \cap C(k)$ is contained in the preimage of $\{\xi : |\xi| \le |\theta_1|\}$.
  Otherwise, write $c' = \gamma^2$ with some $\gamma \in k^\times$. Taking
  \[ z = \frac{1}{2} \Bigl(x + \frac{y}{\gamma h(x)}\Bigr) \,, \]
  we can parameterize $\pi^{-1}(D)$ via
  \[ z \longmapsto \Bigl(z + \frac{a}{4z},
                         \gamma \cdot \Bigl(z - \frac{a}{4z}\Bigr)
                                   h\Bigl(z + \frac{a}{4z}\Bigr)\Bigr) \,.
  \]
  The condition $|z + a/(4z)| < 1$ translates into $|a| < |z| < 1$, which defines
  the annulus~$A$. The covering map to~$D$ is given by $z \mapsto z + (a/4)/z$.
  The involution $z \mapsto a/(4z)$ fixes the $x$-coordinate and changes the sign
  of the $y$-coordinate, so it is the hyperelliptic involution on~$A$.
\end{proof}

\begin{lemma} \label{L:annulus}
  Let $\varphi \colon A_{\alpha,k} \stackrel{\isom}{\to} A \subset \PP^1_k$
  be an open annulus with $A(\C_p) \cap \Theta = \emptyset$ and $A(k) \neq \emptyset$.
  The complement of~$A$ in~$\PP^1_k$ is a disjoint union of two closed disks,
  which partitions~$\Theta$ into two disjoint subsets $\Theta_0$ and~$\Theta_\infty$.
  This induces a factorization $f(x) = c f_0(x) f_\infty(x)$ with $f_0$ and~$f_\infty$
  monic such that the roots of~$f_0$ are the elements of~$\Theta_0$ and the roots
  of~$f_\infty$ are the elements of~$\Theta_\infty$.
  \begin{enumerate}[\upshape(1)]\addtolength{\itemsep}{1mm}
    \item \label{Lann1}
          If $\#\Theta_0$ and~$\#\Theta_\infty$ are odd, we assume in addition
          that $\alpha = |\beta_1|^2$ for some $\beta_1 \in k$ and that
          $r |c f_\infty(\xi)| = |\beta_2|^2$ for some $\xi \in k$ in the closed disk
          defining~$\Theta_0$ and some $\beta_2 \in k$, where $r$ is the outer (or inner) radius
          of~$A$ in terms of some coordinate on~$\PP^1$ such that $0, \infty \notin A$.
          Then $\pi^{-1}(A)$ is an annulus~$A'$ in~$C$. In terms of suitable coordinates
          on $A'$ and~$A$, $\pi \colon A' \to A$ is given by $z \mapsto z^2$
          and the hyperelliptic involution on~$A'$ is $z \mapsto -z$.
    \item \label{Lann2}
          If $\#\Theta_0$ and~$\#\Theta_\infty$ are even, then let $\xi \in A(k)$.
          If $c f_\infty(\xi)$ is not a square in~$k$, then $\pi^{-1}(A) \cap C(k)$
          is empty. Otherwise $\pi^{-1}(A)$ is a disjoint union of two annuli in~$C$,
          each isomorphic to~$A$ via~$\pi$.
  \end{enumerate}
\end{lemma}

\begin{proof}
  We can assume that $\varphi = \id$ and $A = A_{\alpha,k}$ and that $0 \notin \Theta$.
  We fix notations by setting $\Theta_0 = \{\theta \in \Theta : |\theta| \le \alpha\}$
  and $\Theta_\infty = \{\theta \in \Theta : |\theta| \ge 1\}$. Note that, in a similar
  way as in the proof of Lemma~\ref{L:disk}, the sets $\Theta_0$ and~$\Theta_\infty$
  are each fixed by the action of the absolute Galois group of~$k$. In particular,
  the product $\prod_{\theta \in \Theta_\infty} (-\theta)$ is in~$k$.
  We can then write the equation of~$C$ on~$A$ as
  \[ y^2 = c \prod_{\theta \in \Theta_\infty} (-\theta) \cdot x^{\#\Theta_0}
             \Bigl(\prod_{\theta \in \Theta_\infty} f_\theta^-(x)
                   \prod_{\theta \in \Theta_0} f_\theta^+(x)\Bigr)^2
         = c' x^{\#\Theta_0} h(x)^2 \,,
  \]
  where $c' = c \prod_{\theta \in \Theta_\infty} (-\theta) = (-1)^{\#\Theta_\infty} c f_\infty(0)$ and
  $h(x) = \prod_{\theta \in \Theta_\infty} f_\theta^-(x)
            \prod_{\theta \in \Theta_0} f_\theta^+(x)$.

  In case~\eqref{Lann1}, writing $u c' = \gamma^2$ with $\gamma \in k$ and $u \in \calO^\times$,
  we obtain the parameterization
  \[ A_{\sqrt{\alpha},k} \stackrel{\isom}{\To} A' = \pi^{-1}(A), \qquad
     z \longmapsto \bigl(u z^2,
                         \gamma u^{(\#\Theta_0-1)/2} z^{\#\Theta_0} h(u z^2)\bigr) \,,
  \]
  in a similar way as in case~\eqref{Ldisk2} of Lemma~\ref{L:disk}.
  The statements on~$\pi$ and on the hyperelliptic involution follow in the same
  way as there.
  We remark that the condition `$r |c f_\infty(\xi)| = |\beta_2|^2$ for some $\beta_2 \in k$'
  is (assuming that $\alpha = |\beta_1|^2$ for some $\beta_1 \in k$) invariant under
  coordinate transformations.

  In case~\eqref{Lann2} we have $x^{\#\Theta_0} h(x)^2 = \bigl(x^{\#\Theta_0/2} h(x)\bigr)^2$.
  This case is similar to case~\eqref{Ldisk1} of Lemma~\ref{L:disk}: there are no solutions
  in~$k$ unless $c' = \gamma^2$ is a square, and in the latter case, we have
  $\pi^{-1}(A) = A^+ \cup A^-$ with
  \[ A^\pm = \bigl\{(\xi, \pm \gamma \xi^{\#\Theta_0/2} h(\xi)) : \xi \in A\bigr\} \,. \qedhere \]
\end{proof}

We state some results on involutions of disks and annuli.
An \emph{analytic involution} on~$D_{0.k}$ or~$A_{\alpha,k}$ is an analytic automorphism~$\iota$
of order two. Recall that we assume the residue characteristic~$p$ to be odd.
We do not claim that the results below are original, but we were unable to find
a suitable reference.

\begin{lemma} \label{L:invol}
  Let $0 < \alpha < 1$ be of the form $\alpha = |\xi|$ for some $\xi \in k$.
  \begin{enumerate}[\upshape (1)]\addtolength{\itemsep}{1mm}
    \item \label{Linvol1}
          Let $\iota \colon D_{0,k} \to D_{0,k}$ be an analytic involution.
          Then $\iota$ has a unique fixed point in~$D_0(\C_p)$, which is in fact in~$D_0(k)$,
          and in terms of a suitable coordinate~$u$ on~$D_{0,k}$, $\iota$ is given by $u \mapsto -u$.
    \item \label{Linvol2}
          Let $\iota \colon A_{\alpha,k} \to A_{\alpha,k}$ be an analytic involution
          such that $|\iota(\xi)| = |\xi|$ for all $\xi \in A_\alpha(\C_p)$.
          Then $\iota$ is given in terms of a suitable coordinate~$u$ on~$A_{\alpha,k}$
          by $u \mapsto -u$. In particular, $\iota$~has no fixed points,
          and $A_\alpha/\langle \iota \rangle \isom A_{\alpha^2}$ is an annulus.
    \item \label{Linvol3}
          Let $\iota \colon A_{\alpha,k} \to A_{\alpha,k}$ be an analytic involution
          such that $|\iota(\xi)| = \alpha/|\xi|$ for all $\xi \in A_\alpha(\C_p)$.
          Then $\iota$ is given in terms of a suitable coordinate~$u$ on~$A_{\alpha,k}$
          by $u \mapsto a/u$ for some $a \in k$ with $|a| = \alpha$.
          In particular, $\iota$ has exactly two fixed points $u = \pm\sqrt{a}$ in~$A_\alpha(\C_p)$,
          $A_{\alpha,k}/\langle \iota \rangle \isom D_{0,k}$ is a disk,
          and the covering $A_{\alpha,k} \to D_{0,k}$ is branched above two points.
  \end{enumerate}
\end{lemma}

\begin{proof} \strut
  \begin{enumerate}[\upshape (1)]%
    \addtolength{\itemsep}{1mm}\setlength{\parindent}{0mm}\setlength{\parskip}{1ex plus 0.5ex}
    \item We first show that we can assume that $0$ is a fixed point of~$\iota$,
          possibly after a coordinate change.
          So assume otherwise. $\iota$ is then given by a power series
          $\sum_{n=0}^\infty a_n z^n \in \calO\pws{z}$
          whose constant term satisfies $0 < |a_0| < 1$.
          Since $\iota$ is an involution, we have
          $\iota(a_0) = 0$, which implies that $|a_0 + a_1 a_0| \le |a_0|^2$,
          so $|a_1 + 1| \le |a_0|$. This in turn implies that
          $z \mapsto (\iota(z) + z)/2$ is contracting on every sufficiently large closed disk
          contained in~$D_0(k)$ and so has a fixed point in~$D_0(k)$ by the Banach
          fixed point theorem. We can then shift the coordinate so that
          the fixed point is at the origin; we denote this coordinate by~$z$ again.

          We now have $\iota(0) = 0$. Then $\iota(z) = a_1 z + a_2 z^2 + \ldots$,
          and $\iota \circ \iota = \id$ implies that $a_1^2 = 1$. If $a_1 = 1$,
          then it follows that $\iota = \id$, which is excluded: assume otherwise, then
          $\iota(z) = z + \beta z^n + \ldots$ with $n > 1$ and $\beta \neq 0$,
          which leads to the contradiction
          \[ z = \iota(\iota(z)) = z + 2 \beta z^n + \ldots \,. \]
          So $a_1 = -1$ and $\iota(z) = -z (1 + h(z))$ with $h \colon D_{0,k} \to D_{0,k}$.
          Write $h_1(z) = 1 + h(z)$.
          The relation $\iota(\iota(z)) = z$ implies that $h_1(z) h_1(\iota(z)) = 1$.
          We set $u(z) = z (1 + h_1(z))/2$; then $u$ is a coordinate on~$D_{0,k}$, and
          \[ u(\iota(z)) = \iota(z) \frac{1 + h_1(\iota(z))}{2}
                         = -z \frac{h_1(z) + h_1(z) h(\iota(z))}{2}
                         = -z \frac{h_1(z) + 1}{2}
                         = -u(z) \,,
          \]
          so in terms of~$u$, $\iota$ acts as $u \mapsto -u$.
          In particular, $0$ is the only fixed point of~$\iota$ on~$D_0(\C_p)$.
    \item Since $|\iota(\xi)| = |\xi|$, $\iota$ is given by a Laurent series
          $z f_1(z)$ where $|f_1(\xi)| = 1$ for all $\xi \in A_\alpha(\C_p)$.
          Let $a_0$ be the constant term of~$f_1(z)$; then $|a_0| = 1$ and
          $|f_1(\xi) - a_0| < 1$ for all $\xi \in A_\alpha(\C_p)$.
          Writing $\lambda = a_0$ and $h(z) = f_1(z)/\lambda - 1$, we have
          $\iota(z) = \lambda z (1 +  h(z))$ with $h \colon A_{\alpha,k} \to D_{0,k}$.
          Since $\iota$ is an involution, we find that $|\lambda^2 - 1| < 1$.
          If $\lambda$ were close to~$1$, then $|f(\xi) - \xi| < |\xi|$,
          so $\iota$ would induce an involution of the open disk
          $\{|z-\xi| < |\xi|\} \subset A_\alpha$, for each $\xi \in A_\alpha(\C_p)$.
          By part~\eqref{Linvol1}, $\iota$ would have a fixed point in each
          of these disks, which is impossible, since there are infinitely many of them
          (even with fixed~$|\xi|$)
          and $\iota$ is not the identity. It follows that $|\lambda + 1| < 1$.
          This already implies that $\iota$ has no fixed points.

          Write $h_1(z) = 1 + h(z)$.
          Note that $\iota(\iota(z)) = z$ implies that $h_1(z) h_1(\iota(z)) = \lambda^{-2}$.
          Similarly as in part~\eqref{Linvol1} we set
          $u(z) = z (1 - \lambda h_1(z))/(1-\lambda)$; this is a coordinate on~$A_{\alpha,k}$.
          We check that
          \[ u(\iota(z)) = \iota(z) \frac{1 - \lambda h_1(\iota(z))}{1 - \lambda}
                         = \lambda z \frac{h_1(z) - \lambda h_1(z) h_1(\iota(z))}{1 - \lambda}
                         = z \frac{\lambda h_1(z) - 1}{1 - \lambda}
                         = -u(z)
          \]
          as before. The last claim is then clear.
    \item Here we have $\iota(z) = \frac{a}{z} (1 + h(z))$ for some $a \in k$ with $|a| = \alpha$
          and some analytic map $h \colon A_{\alpha,k} \to D_{0,k}$. Write $h_1(z) = 1 + h(z)$
          The fact that $\iota$ is an involution implies this time that $h_1(\iota(z)) = h_1(z)$.
          Set $u(z) = z h_1(z)^{-1/2}$. Then
          \[ u(\iota(z)) = \iota(z) h_1(\iota(z))^{-1/2}
                         = \frac{a}{z} h_1(z) h_1(z)^{-1/2}
                         = \frac{a}{z h_1(z)^{-1/2}}
                         = \frac{a}{u(z)} \,,
          \]
          so $u$ is a suitable coordinate. The fixed points are where $u^2 = a$;
          the map $A_{\alpha,k} \to D_{0,k}$, $u \mapsto u + a/u$, realizes the quotient
          by~$\langle \iota \rangle$. \qedhere
  \end{enumerate}
\end{proof}

Now we show that every (maximal) annulus in~$C$ arises as in Lemmas \ref{L:disk}
and~\ref{L:annulus}. A \emph{maximal annulus} in~$C$ is an annulus that is not
contained in a disk or in a strictly larger annulus in~$C$.

\begin{proposition} \label{P:no-others}
  Let $\varphi \colon A_{\alpha,k} \stackrel{\isom}{\to} A \subset C$ be a maximal annulus
  such that $A(k) \neq \emptyset$. Then $A$ is obtained from a disk
  or an annulus in~$\PP^1_{k}$ as in Lemma~\ref{L:disk},~\eqref{Ldisk3}, or
  Lemma~\ref{L:annulus}, \eqref{Lann1} or~\eqref{Lann2}. In the latter two cases,
  the two sets $\Theta_0$ and~$\Theta_\infty$ both have at least three elements.
\end{proposition}

\begin{proof}
  We consider the action of the hyperelliptic involution~$\iota$ on~$A$
  and its pullback $\varphi^* \iota$ to~$A_{\alpha,k}$.
  There are three possibilities.
  \begin{enumerate}[(1)]\addtolength{\itemsep}{1mm}
    \item $A \cap \iota(A) = \emptyset$. \\[1mm]
          Then clearly $\pi(A)$ is analytically
          isomorphic to~$A$, hence is an annulus in~$\PP^1_k$
          that does not contain any branch points of~$\pi$. We must then be
          in case~\eqref{Lann2} of Lemma~\ref{L:annulus}, since the preimage
          of~$\pi(A)$ splits into the two annuli $A$ and~$\iota(A)$.
          If $\Theta_0$ or~$\Theta_\infty$
          had zero or two elements, then we could `fill in' the annulus~$\pi(A)$ to
          obtain a disk containing zero or two branch points. Then $A$ would be contained in
          a disk or in a larger annulus by Lemma~\ref{L:disk}, \eqref{Ldisk1} or~\eqref{Ldisk3},
          a contradiction.
    \item $\iota(A) = A$ and $\iota$ preserves the orientation of the chain corresponding to~$A$.
          \\[1mm]
          Let $\xi \in A_\alpha(\C_p)$. There is a finite extension~$K$ of~$k$ such
          that $|\xi| = |\beta|$ for some $\beta \in K$. Since the minimal regular
          model~$\calC$ of~$C$ is semistable near the reduction of~$A$, the special fiber of
          the minimal regular model of~$C$ over~$\calO_K$ contains a chain corresponding
          to~$A_K$ obtained by successive blow-ups of intersection points of components
          of the chain in~$\calC_s$ with other components (within or outside the chain).
          There is one such component that corresponds to the points $\xi' \in A_\alpha(K)$
          with $|\xi'| = |\beta| = |\xi|$. Since $\iota$ preserves the orientation
          of the chain, it fixes every component; it follows that $|\varphi^*\iota(\xi)| = |\xi|$.
          By Lemma~\ref{L:invol},~\eqref{Linvol2},
          $\pi(A)$ is an annulus of height~$\alpha^2$ in~$\PP^1_k$ that does not contain
          any branch points. We must then be in case~\eqref{Lann1} of Lemma~\ref{L:annulus}, since
          $A \to \pi(A)$ is an unramified double cover.
          (The condition $r |c f_\infty(\xi)| = |\beta_2|^2$ is automatically satisfied,
          since $A$ is an annulus over~$k$.)
          If $\Theta_0$ or~$\Theta_\infty$ had only one element, then we could again
          `fill in' the annulus~$\pi(A)$ to obtain a disk containing exactly one branch point,
          so that $A$ would be contained in a disk in~$C$ by Lemma~\ref{L:disk},~\eqref{Ldisk2}.
    \item $\iota(A) = A$ and $\iota$ reverses the orientation of the chain corresponding to~$A$. \\[1mm]
          By a similar argument as in the preceding case, we see that
          $|\varphi^* \iota(\xi)| = \alpha/|\xi|$. By Lemma~\ref{L:invol},~\eqref{Linvol3},
          $\pi(A)$ is a disk in~$\PP^1_k$ that contains exactly two branch points.
          We must then be in case~\eqref{Ldisk3} of Lemma~\ref{L:disk}.
          \qedhere
  \end{enumerate}
\end{proof}

We give names to the three possible kinds of annuli.

\begin{definition}
  Let $A$ be a maximal annulus in~$C$. We call $A$
  \begin{enumerate}[(1)]\addtolength{\itemsep}{1mm}
    \item a \emph{branch annulus}, if $A$ is obtained as in Lemma~\ref{L:disk},~\eqref{Ldisk3};
    \item an \emph{odd annulus}, if $A$ is obtained
          as in Lemma~\ref{L:annulus},~\eqref{Lann1} (with $\#\Theta_0, \#\Theta_\infty \ge 3$);
    \item an \emph{even annulus}, if $A$ is obtained
          as in Lemma~\ref{L:annulus},~\eqref{Lann2} (with $\#\Theta_0, \#\Theta_\infty \ge 4$).
  \end{enumerate}
\end{definition}

Now we describe what the regular differentials of~$C$ look like on the various
types of annuli.

\begin{proposition} \label{P:differentials}
  Let $\varphi \colon A_{\alpha,k} \to C$ be an annulus in~$C$,
  in terms of a coordinate~$z$ as in the proofs of
  Lemma~\ref{L:disk},~\eqref{Ldisk3} and Lemma~\ref{L:annulus}.
  Then there is an analytic function $h \colon A_{\alpha,k} \to D_{0,k}$
  such that a basis of $\varphi^* \Omega^1_C(k)$ is given by
  \[ \Bigl(z + \frac{a}{4z}\Bigr)^\nu (1+h(z))\,\frac{dz}{z}\,, \qquad
      z^{2\nu+2-\#\Theta_0} (1+h(z))\,\frac{dz}{z}\,, \qquad
      z^{\nu+1-\#\Theta_0/2} (1+h(z))\,\frac{dz}{z}\,,
  \]
  for $\nu = 0, 1, 2, \ldots, g-1$, when $A$ is a branch, odd, or even annulus, respectively.
\end{proposition}

\begin{proof}
  This is an easy consequence of the parameterizations and the fact that
  $\Omega^1_C(k)$ is spanned by $x^\nu\,dx/y$ for $\nu = 0, 1, 2, \ldots, g-1$.
\end{proof}

\begin{corollary} \label{C:diffann}
  Let $\varphi \colon A_{\alpha,k} \to C$ be an annulus as before.
  There are numbers $n_1 < 0 < n_2$ with $n_2 - n_1 \le 2g-2$
  and an analytic map $h \colon A_{\alpha,k} \to D_{0,k}$ such that
  for every differential $\omega \in \Omega^1_C(k)$, we have
  \[ \varphi^* \omega = u(z) (1+h(z))\,\frac{dz}{z} \]
  where $u(z) \in k[z,z^{-1}]$ is a Laurent polynomial
  all of whose terms have exponents between $n_1$ and~$n_2$, inclusive.
\end{corollary}

\begin{proof}
  In Proposition~\ref{P:differentials}, we can take $n_1 = -(g-1)$, $n_2 = g-1$ in the branch case,
  $n_1 = 2-\#\Theta_0$, $n_2 = 2g-\#\Theta_0$ in the odd case, and
  $n_1 = 1-\#\Theta_0/2$, $n_2 = g-\#\Theta_0/2$ in the even case. Note that
  $3 \le \#\Theta_0 \le 2g-1$, which implies that $n_1 < 0 < n_2$.
\end{proof}

A bound like this for the `relevant' exponents is important to obtain uniform bounds.
We note that Katz, Rabinoff and Zureick-Brown~\cite{KRZB}*{Lemma~4.15}
prove the following statement that applies to arbitrary curves, but has a weaker conclusion.
It is this extension that allows them to generalize our results from
hyperelliptic to arbitrary curves.

\begin{proposition}[Katz, Rabinoff, Zureick-Brown] \label{P:diffKRZB}
  Let $k$ be a $p$-adic field ($p = 2$ is allowed here).
  Let $C$ be a curve over~$k$ of genus~$g \ge 2$ and let $\varphi \colon A_{\alpha,k} \to C$
  be an annulus. Then for every differential $\omega \in \Omega^1_C(k)$, we have
  \[ \varphi^* \omega = u(z) (1+h(z))\,\frac{dz}{z} \]
  where $h \colon A_{\alpha,k} \to D_{0,k}$ and $u(z) \in k[z,z^{-1}]$ is a Laurent polynomial
  all of whose terms have exponents between $-(2g-2)$ and~$2g-2$, inclusive.
\end{proposition}

It would be interesting to see whether the conclusion can be strengthened
as in Corollary~\ref{C:diffann}.

%%%%%%%%%%%%%%%%%%%%%%%%%%%%%%%%%%%%%%%%%%%%%%%%%%%%%%%%%%%%%%%%%%%%%%%%%%%

\section{The pull-back of an abelian logarithm to an annulus} \label{S:integ}

We fix a basepoint $P_0 \in C(k)$; this gives rise to the embedding
$i \colon C \to J$, $P \mapsto [P-P_0]$, defined over~$k$.
Let $\omega$ be a regular differential on~$C$ and denote by~$\omega_J$
the corresponding regular and invariant $1$-form on~$J$ (so that $\omega = i^* \omega_J$).
We write for $P \in C(k)$
\[ \lambda_\omega(P) = \oint_{P_0}^P \omega
                     = \oint_O^{[P-P_0]} \omega_J
                     = \langle \omega_J, \log_J {[P-P_0]} \rangle \,.
\]
If $\varphi \colon D_0 \to C$ is an open disk in~$C$, then
\[ \varphi^* \omega = w(z)\,dz \]
with an analytic regular function $w(z)$ on~$D_0$. Let $\ell$ be a power
series whose derivative is~$w$. Then it is well-known that for
$\xi_0, \xi_1 \in D_0(k)$ we have
\[ \oint_{\varphi(\xi_0)}^{\varphi(\xi_1)} \omega
     = \int_{\xi_0}^{\xi_1} w(z)\,dz
     = \ell(\xi_1) - \ell(\xi_0) \,.
\]
Using Newton polygons, one then shows (see for example~\cite{Stoll2006-chabauty}*{Section~6})
that the number of zeros of~$\lambda_\omega$ on~$\varphi(D_0(k))$
(or even $\varphi(D_0(k^\unr))$)
is bounded by $1$ plus the number~$n$ of zeros of~$\omega$ (counted with multiplicity)
on~$\varphi(D_0(\bar{k}))$ plus a term,
denoted by $\delta(v,n)$ in~\cite{Stoll2006-chabauty},
that depends only on~$n$, $p$ and the ramification
index~$e$ of~$k$. We write $\Delta_k(s,r)$ for what is denoted~$\Delta_v(s,r)$
in~\cite{Stoll2006-chabauty}, namely
\[ \Delta_k(s,r) = \max\Bigl\{\sum_{j=1}^s \delta(v,m_j)
                                : m_j \ge 0, \sum_{j=1}^s m_j \le r\Bigr\} \,.
\]
Recall that $e$ denotes the ramification index of~$k$.
If $p > e+1$, then we set
\[ \mu = 1 + \frac{e}{p-e-1} = \frac{p-1}{p-e-1}\,; \]
note that $1 < \mu \le e+1$. By~\cite{Stoll2006-chabauty}*{Lemma~6.2},
we have $\Delta_k(s, n) \le e \lfloor n/(p-e-1) \rfloor$ in this case,
so that $n + \Delta_k(s, n) \le \mu n$.
We have the following bound.

\begin{lemma} \label{L:points-D}
  Let $V \neq 0$ be a linear subspace of codimension~$r$ of the space of regular
  differentials on~$C$ and let $N_D$ denote the number of disks whose union is~$C_D(k)$.
  Then the functions~$\lambda_\omega$ for $\omega \in V$ have at most
  \[ N_D + 2r + \Delta_k(N_D, 2r) \]
  common zeros in~$C_D(k)$. If $p > e + 1$, then we can take the bound to be
  \[ N_D + 2 \mu r \le (5q+2) (g-1) - 3 q (t-1) + 2 \mu r \,. \]
\end{lemma}

\begin{proof}
  This is essentially \cite{Stoll2006-chabauty}*{Theorem~6.6},
  using~\cite{Katz-Zureick-Brown}*{Theorem~4.4} in the case of bad reduction.
  In the case $p > e+1$ we use the bound stated above;
  the bound for~$N_D$ comes from Proposition~\ref{P:DAbounds}.
\end{proof}

Now we consider the situation for an
annulus~$\varphi \colon A_{\alpha,k} \stackrel{\isom}{\to} A \subset C$.
Pulling back~$\omega$, we obtain
\[ \varphi^* \omega = w(z)\,dz = d \ell(z) + c(\omega) \frac{dz}{z} \]
for analytic regular functions $w$ and~$\ell$ on~$A_{\alpha,k}$ and some constant
$c(\omega) \in k$. Let $\Log_0$ denote the branch of the $p$-adic logarithm
that takes the value~$0$ at~$p$. Then, given this choice,
we can define a $p$-adic integral on~$A_\alpha$ by
\[ \int_{\xi_0}^{\xi_1} \varphi^* \omega
     = \int_{\xi_0}^{\xi_1} w(z)\,dz
     \colonequals \bigl(\ell(\xi_1) + c(\omega) \Log_0(\xi_1)\bigr)
                    - \bigl(\ell(\xi_0) + c(\omega) \Log_0(\xi_0)\bigr) \,.
\]
We want to compare this with
\[ \oint_{\varphi(\xi_0)}^{\varphi(\xi_1)} \omega \,. \]
Perhaps surprisingly, these two integrals can differ.

\begin{remark}
  There is in fact a unique $p$-adic integration theory in a suitable sense
  that is functorial and satisfies $\int_1^\xi dz/z = \Log_0(\xi)$ on
  any annulus containing $1$ and~$\xi$~\cite{Berkovich2007}.
  It is called the \emph{Berkovich-Coleman integral} in~\cite{KRZB} to
  distinguish it from the \emph{Abelian integral} that we denote $\oint$ in this paper.
\end{remark}

The following result is crucial. It was first suggested by numerical computations
and appears to be new. Recall that $v \colon k^\times \to \Q$ denotes the
valuation on~$k$, normalized so that $v(p) = 1$.

\begin{proposition} \label{P:integration-A}
  Let $\omega$ and $\varphi \colon A_{\alpha,k} \to C$ be as above, and write
  \[ \varphi^* \omega = d \ell(z) + c(\omega) \frac{dz}{z} \,. \]
  Then there is a constant~$a(\omega) \in k$ depending
  linearly on~$\omega$ such that for $\xi_0, \xi_1 \in A_{\alpha}(k)$ we have
  \begin{align*}
    \oint_{\varphi(\xi_0)}^{\varphi(\xi_1)} \omega
      &= \bigl(\ell(\xi_1) + c(\omega) \Log_0(\xi_1) + a(\omega) v(\xi_1)\bigr)
           - \bigl(\ell(\xi_0) + c(\omega) \Log_0(\xi_0) + a(\omega) v(\xi_0)\bigr) \\
      &= \int_{\xi_0}^{\xi_1} \varphi^* \omega
            + a(\omega) \bigl(v(\xi_1) - v(\xi_0)\bigr) \,.
  \end{align*}
\end{proposition}

\begin{proof}
  Let $\xi_0 \in A_\alpha(k)$.
  Let $i \colon C \to J$ be the embedding sending $\varphi(\xi_0)$ to~$O$.
  According to~\cite{BoschLuetkebohmert-stabredII}*{Proposition~6.3}, the
  analytic map $i \circ \varphi \colon A_{\alpha,k} \to J$ can be written
  uniquely as
  \[ i(\varphi(z)) = \psi_1(\xi_0^{-1} j(z)) + \psi_2(z)\,, \]
  where $j \colon A_{\alpha,k} \to \Gmk$ is the natural inclusion,
  $\psi_1 \colon \Gmk \to J$ is an analytic group homomorphism and
  $\psi_2 \colon A_{\alpha,k} \to U$ is an analytic map, where $U$ denotes the
  formal fiber of the origin on~$J$ (so that $U(k)$ is the subgroup of
  points reducing to the origin). We write $\omega_J$ for the regular $1$-form on~$J$
  such that $i^* \omega_J = \omega$; $\omega_J$ is translation invariant.
  On~$U$, $\omega_J$ is exact, so $\omega_J = d\lambda$ for some analytic
  function~$\lambda$ on~$U$; we can assume $\lambda(O) = 0$.
  The pull-back $\psi_1^* \omega_J$ is a translation
  invariant differential on~$\Gmk$, so it has the form $c\,dz/z$ for some $c \in k$;
  the (multiplicative) translation by~$\xi_0^{-1}$ does not change it.
  The pull-back $\psi_2^* \omega_J$ is $\psi_2^* d\lambda = d(\lambda \circ \psi_2)$.
  Since
  \[ c(\omega)\,\frac{dz}{z} + d\ell(z)
       = \varphi^* \omega
       = \varphi^* i^* \omega_J
       = \psi_1^* \omega_J + \psi_2^* \omega_J
       = c \frac{dz}{z} + d \lambda\bigl(\psi_2(z)\bigr) \,,
  \]
  we see that $\ell(z) = \lambda(\psi_2(z))$ (up to an additive constant) and $c = c(\omega)$.
  Let $\xi_{1} \in A_{\alpha}(k)$. We obtain on the one side that
  \begin{align*}
    \oint_{\varphi(\xi_0)}^{\varphi(\xi_{1})} \omega
      &= \oint_O^{i(\varphi(\xi_{1}))} \omega_J
       = \oint_O^{\psi_1(\xi_0^{-1} \xi_{1})+\psi_2(\xi_{1})} \omega_J \\
      &= \oint_O^{\psi_1(\xi_0^{-1} \xi_{1})} \omega_J
           + \oint_O^{\psi_2(\xi_{1})} d\lambda
       = \oint_O^{\psi_1(\xi_{1}/\xi_0)} \omega_J + \lambda\bigl(\psi_2(\xi_{1})\bigr)
  \end{align*}
  and on the other side that
  \begin{align*}
    \int_{\xi_0}^{\xi_{1}} \varphi^* \omega
      &= \int_{\xi_0}^{\xi_{1}} \Bigl(d\ell(z) + c\frac{dz}{z}\Bigr)
       = \ell(\xi_{1}) - \ell(\xi_0)
          + c \bigl(\Log_0(\xi_{1}) - \Log_0(\xi_0)\bigr) \\
      &= \lambda\bigl(\psi_2(\xi_{1})\bigr) + c \Log_0(\xi_{1}/\xi_0)\,.
  \end{align*}
  Here we use that $\lambda(\psi_2(\xi_0)) = \lambda(O) = 0$.
  So the difference is
  \[ \delta(\xi_{1}/\xi_0)
        \colonequals \oint_{\varphi(\xi_0)}^{\varphi(\xi_{1})} \omega
             - \int_{\xi_0}^{\xi_{1}} \varphi^* \omega
        = \oint_O^{\psi_1(\xi_{1}/\xi_0)} \omega_J - c \Log_0(\xi_{1}/\xi_0) \,.
  \]
  Since $\psi_1$ is a group homomorphism, the first term in the last difference
  is a homomorphism $k^\times \to k$; the same is true for the second term.
  Both terms in the first difference agree on the residue disk~$U_1$ of~$1$,
  since they are given by the same formal integral on~$U_1$.
  Since $\calO^\times/U_1$
  is torsion and the target group~$k$ is torsion-free, we have $\delta = 0$
  on~$\calO^\times$.
  This implies that $\delta(\xi)$ is a linear function of the valuation~$v(\xi)$, so there
  is $a = a(\omega) \in k$ such that $\delta(\xi) = a v(\xi)$.

  That $a(\omega)$ is linear in~$\omega$ is clear, since $\ell$ (if we set $\ell_0 = 0$),
  $c(\omega)$ and the left-hand side are.
\end{proof}

\begin{remark}
  The numerical example mentioned above shows that it is possible to have
  $a(\omega) \neq 0$ and $c(\omega) = 0$, so that the appearance of~$a(\omega)$
  cannot in all cases be avoided by choosing a suitable branch of the $p$-adic
  logarithm.
  In this situation we have $\psi_1^* \omega_J = 0$ and the difference term above
  is given by $\oint_O^{\psi_1(\xi_{1}/\xi_0)} \omega_J$.
  Even though the pull-back of~$\omega_J$
  along~$\psi_1$ vanishes, it does not follow that the abelian integral vanishes
  on the image of~$\psi_1$. Consider for example $\xi_{1}/\xi_0 = p$
  and $P = \psi_1(p) \in J(k)$.
  There is a positive integer $n$ such that $n P \in U$; then
  \[ \oint_O^{\psi_1(p)} \omega_J = \frac{1}{n} \oint_O^{nP} \omega_J
                                  = \frac{1}{n} \lambda_{\omega_J}(nP) \,.
  \]
  There is no reason to assume that $\log_J (nP)$ is parallel to the derivative
  of~$\psi_1$ at~$1$, so $\psi_1^* \omega_J = 0$ does not in general imply
  that $\lambda_{\omega_J}(nP)$ vanishes.
\end{remark}

\begin{remark}
  Katz, Rabinoff and Zureick-Brown~\cite{KRZB} generalize Proposition~\ref{P:integration-A}
  to a comparison of the abelian integral and the Berkovich-Coleman integral on more general
  `wide open' subsets of (the Berkovich analytic space associated to)~$C$.
\end{remark}

We say that $\omega$ is \emph{good} for the annulus $\varphi \colon A_{\alpha,k} \to C$
if both $c(\omega)$ and~$a(\omega)$ in Proposition~\ref{P:integration-A} vanish.
This is a linear condition on~$\omega$ of codimension at most two.

\begin{lemma} \label{L:technical}
  In the situation of Proposition~\ref{P:integration-A} assume that
  $C$ is hyperelliptic and $p$ is odd. Let $V \neq 0$ be a linear subspace
  of codimension~$r \ge 1$ of the space of regular differentials on~$C$.
  Then there exists $0 \neq \omega \in V$ such that
  $\varphi^* \omega = u(z) (1+h(z))\,dz/z$ with an analytic map $h \colon A_{\alpha,k} \to D_{0,k}$
  and a Laurent polynomial~$u$ such that the terms in~$u$
  have exponents between $n_1$ and~$n_2$ (inclusive), where $n_1 \le 0 \le n_2$
  and $n_2 - n_1 \le 2r$ if the annulus is branch or odd, and $n_2 - n_1 \le r$
  if the annulus is even.
\end{lemma}

\begin{proof}
  This follows from Proposition~\ref{P:differentials}.

  In the branch case,
  \[ \varphi^* \omega
      = \sum_{\nu=0}^{g-1} a_\nu \Bigl(z + \frac{a}{4z}\Bigr)^\nu (1 + h(z))\,\frac{dz}{z} \,.
  \]
  Since $V$ has codimension~$r$, we can impose
  $g-1-r$ linear conditions, which we can take to be the vanishing of the
  coefficients $a_{r+1}, a_{r+2}, \ldots, a_{g-1}$. Then the claim holds
  with $n_1 = -r$, $n_2 = r$.

  In the odd case,
  \[ \varphi^* \omega = \sum_{\nu=0}^{g-1} a_\nu z^{2\nu+2-\#\Theta_0} (1 + h(z))\,\frac{dz}{z} \,. \]
  Here we impose the vanishing of $g-1-r$ coefficients~$a_\nu$ with $\nu$ small
  and/or large, so that the remaining coefficients form a contiguous sequence of
  odd integers containing negative as well as positive numbers.
  The difference of the largest and the smallest remaining exponent is then~$2r$.

  In the even case,
  \[ \varphi^* \omega = \sum_{\nu=0}^{g-1} a_\nu z^{\nu+1-\#\Theta_0/2} (1 + h(z))\,\frac{dz}{z} \,. \]
  We proceed in the same way as in the odd case, leaving a contiguous range of
  exponents containing~$0$ and of length~$r$.
\end{proof}

Recall that we fix some $P_0 \in C(k)$ and set
\[ \lambda_\omega \colon C(\C_p) \To \C_p, \qquad P \longmapsto \oint_{P_0}^P \omega \,. \]

\begin{proposition} \label{P:good-A}
  In the situation of Proposition~\ref{P:integration-A} assume that
  $V \neq 0$ is a linear subspace of the space of regular differentials on~$C$
  of codimension~$r \ge 1$ and such that all elements of~$V$ are good.
  Assume further that $C$ is hyperelliptic and that $p$ is odd.
  Then the number of common zeros on $\varphi(A_\alpha(k^\unr))$
  of the~$\lambda_\omega$ for all $\omega \in V$
  is bounded by a number~$B_A(p,e,r)$ that depends only on~$r$, $p$ and the
  ramification index~$e$ of~$k$.

  If $p > e + 1$, then we can take
  $B_A(p,e,r) = 2 \mu r$. If the annulus is even
  and $r \ge 2$, we can replace this by $\mu r$,
  so that we get the bound $2 \mu r$ for the union
  of the annulus and its image under the hyperelliptic involution.
\end{proposition}

\begin{proof}
  By Lemma~\ref{L:technical}, there is $0 \neq \omega \in V$
  such that $\varphi^* \omega = u(z) (1 + h(z))\,dz/z$ with an analytic map
  $h \colon A_{\alpha,k} \to D_{0,k}$ and a Laurent polynomial~$u$ having exponents
  between $n_1$ and~$n_2$ with $n_1 \le 0 \le n_2$ and $n_2 - n_1 \le 2r$ ($\le r$
  if the annulus is even). Since $r \ge 1$ (or $r \ge 2$ in the even case), we
  can in fact assume that $n_1 < 0 < n_2$.
  Given this, the proof can be carried out using Newton polygons
  in essentially the same way as for power series.
  One possibility for this is to consider the `positive' and the `negative'
  part of the formal integral separately. To the positive part, we can directly
  apply the corresponding result for power series; for the negative part, we
  substitute $z \leftarrow z^{-1}$. The bound we obtain for the length of the
  relevant interval of exponents (belonging to segments of the Newton polygon
  corresponding to zeros of absolute values in the largest $k$-defined closed annulus
  contained in~$A_\alpha$) is then $n_2 - n_1 + \Delta_k(2, n_2 - n_1)$, which for $p > e+1$
  can be bounded as stated.
\end{proof}

Note that in contrast to the corresponding result for disks, the
`$1 + {}$' term that causes the non-uniformity of the `classical' Chabauty-Coleman
bound does not show up here. This is because the constant of integration affects
a coefficient whose exponent lies within the relevant part of the Newton polygon
of the formal integral, whereas in the power series case, it can increase the
length of the relevant range of exponents by~$1$.

\begin{corollary} \label{C:points-A}
  Let $V$ be a linear subspace of codimension $r \le g-3$ of the space of regular
  differentials on~$C$, where $C$ is as in Proposition~\ref{P:good-A}.
  Let $t$ be as in Proposition~\ref{P:DAbounds}.
  Then the number of common zeros in~$C_A(k)$ of all $\lambda_\omega$
  for $\omega \in V$ is bounded by
  \[ (2g - 3 + t) B_A(p, e, r+2) \,. \]
  If $p > e + 1$, then we have the bound
  \[ \min\{2g-1, 2g - 3 + t\} \cdot 2 \mu (r+2) \,. \]
\end{corollary}

\begin{proof}
  For each annulus $A$ occurring in~$C_A(k)$,
  we let $V_A$ be the subspace of~$V$ consisting of differentials
  that are good for~$A$. Then $V_A$ has codimension at most $r+2 < g$,
  and by Proposition~\ref{P:good-A}
  the number of common zeros of~$\lambda_\omega$ on~$A$ for $\omega \in V_A$
  is at most~$B_A(p,e,r+2)$. We
  multiply by the bound $2g-3+t$ for the number of annuli from Proposition~\ref{P:DAbounds}
  to obtain the result.

  Now assume that $p > e+1$; then $B_A(p,e,r+2)$ is bounded by the second
  factor in the last formula.
  By the last statement in Proposition~\ref{P:good-A}, we can replace $2g-3+t$
  by a bound for the number of orbits of annuli under the hyperelliptic involution,
  which can be obtained as follows.
  The image of
  a minimal skeleton of the $p$-adic Berkovich analytic space associated to~$C$
  in the Berkovich projective line is a tree with at most~$2g$ nodes (it is obtained
  from the convex hull of the branch points, which is a tree with $2g+2$ leaves,
  by removing the leaves and the edges connected to them) and hence at
  most~$2g-1$ edges. The edges correspond to the orbits of annuli under~$\iota$,
  so there are at most~$2g-1$ such orbits. (These are the orbits we see when
  $C$ has split semistable reduction. Since annuli persist under finite extensions
  of the $p$-adic base field, this gives an upper bound for the orbits of annuli
  that are relevant to us here.)
\end{proof}

%%%%%%%%%%%%%%%%%%%%%%%%%%%%%%%%%%%%%%%%%%%%%%%%%%%%%%%%%%%%%%%%%%%%%%%%%%%

\section{Bounding the number of points mapping into a subgroup of small rank}
\label{S:mainT}

In this section we state and prove our main result.

\begin{theorem} \label{T:local-main}
  Let $k$ be a $p$-adic field with $p$ odd and write $e$ for the ramification
  index of~$k$ and $q$ for the size of its residue field. Let $g \ge 3$ and $0 \le r \le g-3$.
  Then there is a bound $N(k, g, r)$ depending only on~$k$, $g$ and~$r$ such that
  the following holds.

  Let $C \colon y^2 = f(x)$ be a hyperelliptic curve of genus~$g$ over~$k$.
  We denote by~$J$ the Jacobian variety of~$C$. Let $\Gamma \subset J(k)$ be
  a subgroup of rank~$r$.
  Let $i \colon C \to J$ be an embedding given by choosing some basepoint~$P_0 \in C(k)$.
  Then
  \[ \#\{P \in C(k) : i(P) \in \Gamma\} \le N(k, g, r) \,. \]

  If $p > e + 1$, then we can take
  \begin{align*}
    N(k, g, r) &= \bigl(2 + 5q + 4 \mu (r + 2)\bigr) (g-1)
                  + \max\{3q-4\mu, 4\mu (r+1) - 3q\} \\
               &\le \bigl(2 + 5q + 4 \mu (r+2)\bigr) g \,,
  \end{align*}
  where $\mu = (p-1)/(p-e-1) \le e+1$.
\end{theorem}

\begin{proof}
  The rank condition implies that there is a $k$-vector space~$V$
  of regular differentials on~$C$ of codimension $\le r \le g-3$ and
  such that each $\omega \in V$ annihilates~$\Gamma$ under the Chabauty-Coleman pairing.
  This means that (taking~$P_0$ to be the basepoint for~$\lambda_\omega$)
  the set of points in question is contained in the common zero set of
  all~$\lambda_\omega$ for $\omega \in V$. We can then use
  Lemma~\ref{L:points-D} and Corollary~\ref{C:points-A} to bound the
  number of points in~$C_D(k)$ and in~$C_A(k)$, respectively, that map to~$\Gamma$.
  Adding these bounds gives the first result.

  In the case $p > e+1$, adding the corresponding explicit bounds and maximizing
  over $0 \le t \le g$ gives the bound
  \begin{align*}
   (5q+2) (g-1) &+ 3 q + 2 \mu r + (2g - 3) 2 \mu (r+2)
                        + 2 \max\{0, 2 \mu (r+2) - 3q\} \\
                &= \bigl(2 + 5 q + 4 \mu (r + 2)\bigr) (g-1)
                  + \max\{3q-4\mu, 4\mu (r+1) - 3q\} \,. \qedhere
  \end{align*}
\end{proof}

\begin{remark}
  It is conceivable that a more careful analysis of the functions $\lambda_\omega$
  on annuli will result in a bound for the number of zeros that applies to
  differentials~$\omega$ that do not necessarily satisfy the conditions that
  $c(\omega)$ and/or~$a(\omega)$ (in the notation of Proposition~\ref{P:integration-A}) vanish.
  If this is indeed the case, then the condition
  $r \le g-3$ can be relaxed to $r \le g-2$ or even $r \le g-1$.
  However, in view of the facts that $\Log_0(z) = 0$ has infinitely many solutions
  in~$\Q_p$ and that the number of solutions to $z^{-1} + a v(z) + z = 0$
  is unbounded when the valuation of~$a$ can be arbitrarily negative, it is very
  likely that more subtle arguments will be necessary to obtain uniform bounds
  under these less restrictive assumptions.
\end{remark}

\begin{remark} \label{R:variants}
  We sketch two variants of the approach taken here.
  \begin{enumerate}[(i)]\addtolength{\itemsep}{1mm}
    \item One possibility is to prove a result like Theorem~\ref{T:local-main}
          above for \emph{semi-stable} curves. Since a curve of genus~$g$ over
          a $p$-adic field~$k$ acquires semi-stable reduction over an extension
          of~$k$ of degree bounded in terms of~$g$ only, this implies the general
          result. The advantage of this approach is that the structure of the
          special fiber of the minimal regular model is much easier to understand,
          so the discussion of the combinatorics of arithmetic graphs as in
          Section~\ref{S:AG} can be bypassed. The disadvantage is that the
          explicit bounds one obtains are much worse, since one is effectively
          working over much larger fields.
    \item Another possibility is to prove directly that for a given hyperelliptic
          curve~$C$ of genus~$g$ over~$k$, one can partition $\PP^1(k)$
          into $\ll q g$ disks containing at most one branch point and $\ll g$
          disks containing exactly two branch points and annuli containing
          no branch points of the hyperelliptic covering map $\pi \colon C \to \PP^1$
          as in Lemmas \ref{L:disk} and~\ref{L:annulus}.
          Since each of the former gives rise to zero, one or two residue disks
          on~$C$ (when $p$ is odd) and each of the latter gives rise to zero, one or two annuli,
          one obtains a result similar to Proposition~\ref{P:DAbounds}.
          The advantage is again that one circumvents the discussion of arithmetic
          graphs, which, however, has to be replaced by a discussion of partitions
          of~$\PP^1(k)$ as above. A disadvantage of this approach is that it
          is restricted to hyperelliptic curves from the start.
          Another advantage is that with some modifications it also works for $p = 2$.
  \end{enumerate}
  No matter which approach is taken, Proposition~\ref{P:good-A} remains
  the crucial ingredient of the proof.
\end{remark}

%%%%%%%%%%%%%%%%%%%%%%%%%%%%%%%%%%%%%%%%%%%%%%%%%%%%%%%%%%%%%%%%%%%%%%%%%%%

\section{A uniform bound on the number of rational points} \label{S:ratpts}

We can apply the result of the previous section to obtain bounds for
the number of rational points on hyperelliptic curves with small
Mordell-Weil rank relative to the genus.

\begin{theorem} \label{T:main-ratpoints-general}
  Let $g \ge 3$, $d \ge 1$ and $0 \le r \le g-3$. Then there is a bound
  $R(d,g,r)$ depending only on $d$, $g$ and~$r$ such that for any
  hyperelliptic curve~$C$ of genus~$g$ over a number field~$K$
  of degree at most~$d$ such that the Mordell-Weil rank of its Jacobian
  is~$r$, we have $\#C(K) \le R(d,g,r)$.

  If $d = 1$ (hence $K = \Q$), we can take
  \[ R(1,g,0) = 33(g - 1) + 1 \qquad \text{and} \qquad
     R(1,g,r) = 8rg + 33(g - 1) - 1 \quad \text{for $r \ge 1$.}
  \]
\end{theorem}

\begin{proof}
  Fix some odd prime~$p$. Then there are only finitely many possible
  completions~$k$ at places above~$p$ of number fields of degree~$\le d$.
  We take $R(d,g,r)$ to be the maximum of the bounds~$N(k,g,r)$
  of Theorem~\ref{T:local-main} over all these~$k$.

  Let $C$ be a curve as in the statement. If $C(K) = \emptyset$, there
  is nothing to prove. So we can assume that there is some $P_0 \in C(K)$,
  which we use as basepoint for an embedding $i \colon C \to J$.
  We can then apply Theorem~\ref{T:local-main} to $C$ base-changed
  to a completion~$k$ of~$K$ at a place above~$p$ and to
  $\Gamma = J(K) \subset J(k)$.

  To obtain the bound for $d = 1$, we take $k = \Q_3$ (with $p = 3 > 2 = e + 1$
  and $q = p = 3$).
\end{proof}

\begin{remark}
  We note that by choosing $p \approx \sqrt{r}$ for large~$r$ instead of~$p = 3$,
  one obtains a bound with leading term $(4r + O(\sqrt{r})) g$.
\end{remark}

\begin{remark}
  Using the bound in Theorem~\ref{T:local-main} when $p > e+1$, we obtain
  the estimate
  \[ R(d,g,r) \ll g \bigl(p^d + d(r+1)\bigr) \ll g \bigl((2d)^d + d(r+1)\bigr) \]
  where $p$ is the smallest prime $> d+1$. (The worst case is when $K$ is totally
  ramified at all primes $\le d+1$ and inert at all reasonably small primes $> d+1$.)
\end{remark}

Taking $r = 0$, we obtain the following.

\begin{corollary} \label{C:torsion}
  Let $C$ be a hyperelliptic curve of genus $g \ge 3$ over~$\Q$. Then
  any torsion packet on~$C$ can contain at most $33 (g - 1) + 1$ rational points.
\end{corollary}

Recall that a \emph{torsion packet} on~$C$ is a subset of~$C$ such that the
difference of any two points in the set is a torsion point on the Jacobian.

If we write $T(g)$ for the maximal number of rational points in a torsion
packet on a hyperelliptic curve of genus~$g$ over~$\Q$, then this gives
\[ 2 \le \liminf_{g \to \infty} \frac{T(g)}{g}
     \le \limsup_{g \to \infty} \frac{T(g)}{g} \le 33
\]
(the leftmost inequality is obtained by considering curves with all $2g+2$
Weierstrass points rational). So we know that the growth rate of~$T(g)$ is linear!
An analogous statement holds for the size of a set of rational points
mapping into a subgroup of rank $\le r$.

%%%%%%%%%%%%%%%%%%%%%%%%%%%%%%%%%%%%%%%%%%%%%%%%%%%%%%%%%%%%%%%%%%%%%%%%%%%

\section{A uniform version of the Poonen-Stoll result} \label{S:rholog}

Let $C$ be a curve of genus~$g$ over the $p$-adic field~$k$. We fix a $k$-basis
$\uom = (\omega_1, \ldots, \omega_g)$ of the space of regular differentials on~$C$
defined over~$k$. We also fix a point $P_0 \in C(k)$. As
in~\cite{PoonenStoll2014}, we write $\rho$ for the partially
defined composition
\[ k^g \dashedarrow k^g \setminus \{0\} \To \PP^{g-1}(k) \To \PP^{g-1}(\kappa) \]
(recall that $\kappa$ denotes the residue field of~$k$). We define the map
\[ {\log_{\uom}} \colon C(k) \To k^g, \quad
   P \longmapsto \oint_{P_0}^P \underline{\omega} \,;
\]
then we have the partially (away from the finitely many points mapping to torsion
under the embedding of $C$ into~$J$ given by the base-point~$P_0$) defined composition
\[ \rholog_{\uom} \colon C(k) \dashedarrow \PP^{g-1}(\kappa) \,. \]
For a subset~$X$ of~$C(k)$, we write $\rholog_{\uom}(X)$ for the image of the
subset of~$X$ consisting of elements on which $\rholog_{\uom}$ is defined.

We now specialize to $k = \Q_2$.

\begin{lemma} \label{L:rhologA}
  Let $\varphi \colon A_{\alpha,\Q_2} \to C$ be an annulus.
  Then with the notation introduced above, we have
  \[ \#\rholog_{\uom}(\varphi(A_{\alpha}(\Q_2))) \le 96 (g-1) + 31 \,. \]
\end{lemma}

\begin{proof}
  We first need a version of~\cite{PoonenStoll2014}*{Proposition~3.8}
  for Laurent series. So let $\underline{\ell}$ and $\underline{w}$ be
  tuples of Laurent series with coefficients in~$\Q_2$, converging on~$A_{\alpha}$
  and with $d\underline{\ell}(z)/dz = \underline{w}(z)$.
  By Proposition~\ref{P:diffKRZB},
  a linear combination $\sum b_j w_j(z)$, with $(b_j)$ a $\Z_2$-basis
  of the ring of integers of a suitable unramified extension of~$\Q_2$,
  can be written in the form $u(z) h(z)$, where $|h(\xi) - 1| < 1$ for all
  $\xi \in A_{\alpha}(\C_2)$ and $u$ is a Laurent polynomial
  with exponents contained in $[-2g+1, 2g-3]$. Then
  \begin{equation} \label{E:laurbound}
    \#\rho(\underline{\ell}(A_{\alpha}(\Q_2))) \le 12 (g-1) + 3 \,.
  \end{equation}
  This can be proved in the same way as~\cite{PoonenStoll2014}*{Proposition~3.8};
  the point is that the relevant range of exponents of~$\underline{\ell}$
  is contained in $[-2g+2 - \delta(v, 2g-3), 2g-2 + \delta(v, 2g-3)]$.
  We also use $\delta(v, n) \le 1 + n/2$. (This is also analogous to the proof
  of Proposition~\ref{P:good-A}.)

  Write $\varphi^*\uom = d \underline{\ell}(z) + \underline{c}\frac{dz}{z}$
  with $\underline{c} \in \Q_2^g$; we can assume that the constant term
  in~$\underline{\ell}(z)$ is zero. Let $\underline{a} = (a_1, \ldots, a_g)$
  with $a_j = a(\omega_j)$ be the constants arising in Proposition~\ref{P:integration-A}.
  Then, by the same proposition, we have
  \[ \log_{\uom}(\varphi(\xi))
        = \underline{\ell}(\xi) + \underline{c} \Log_0(\xi) + \underline{a} v(\xi) + \underline{b}
  \]
  with a constant vector~$\underline{b}$.
  Let $r \colon \Q_2^g \setminus \{0\} \to \F_2^g \setminus \{0\}$ be the map
  that first scales its argument by a power of~$2$ so that
  its entries are coprime elements of~$\Z_2$ and then reduces it mod~$2$
  (so that $\rho$ is $r$ followed by the canonical map
  $\F_2^g \setminus \{0\} \to \PP^{g-1}(\F_2)$). Since the size of
  $\#\rholog_{\uom}(\varphi(A_{\alpha}(\Q_2)))$ depends only
  on the $\Z_2$-module generated by~$\uom$,
  we are free to replace $\uom$ by any other $\Z_2$-basis of this module.
  We can choose a basis such that all of $\underline{a}$, $\underline{b}$ and~$\underline{c}$
  are of the form $(\ast,\ast,\ast,0,\ldots,0)$.
  We assume in the following that $\underline{a}$, $\underline{b}$ and~$\underline{c}$
  are linearly independent. (If the dimension of their span is strictly less than~$3$,
  an argument similar to that carried out below results in a better bound.)
  For any given $\xi \in A_{\alpha}(\Q_2)$
  such that $\rholog_{\uom}(\varphi(\xi))$ is defined,
  we then have that $r(\log_{\uom}(\varphi(\xi)))$ is of the form
  $(\beta_1,\beta_2,\beta_3,0,\ldots,0)$ with
  $(\beta_1,\beta_2,\beta_3) \in \F_2^3 \setminus \{(0,0,0)\}$
  or $(\beta_1,\beta_2,\beta_3,\lambda_4,\ldots,\lambda_g)$
  with $(\beta_1,\beta_2,\beta_3) \in \F_2^3$,
  where $(0,0,0,\lambda_4, \ldots, \lambda_g) = r(0,0,0,\ell_4(\xi),\ldots,\ell_g(\xi))$.
  This shows that
  \begin{equation} \label{E:bound-rholog}
    \#\rholog_{\uom}(\varphi(A_{\alpha}(\Q_2)))
      \le 8\#\{r(0,0,0,\ell_4(\xi),\ldots,\ell_g(\xi))
                 : \xi \in A_{\alpha}(\Q_2), (\ell_4,\ldots,\ell_g)(\xi) \neq 0\} + 7 \,.
  \end{equation}
  Now \eqref{E:laurbound}, applied to $(\ell_4, \ldots, \ell_g)$, gives
  \[ \#\{r(0,0,0,\ell_4(\xi),\ldots,\ell_g(\xi))
                 : \xi \in A_{\alpha}(\Q_2), (\ell_4,\ldots,\ell_g)(\xi) \neq 0\}
        \le 12(g-1) + 3 \,.
  \]
  Using this in~\eqref{E:bound-rholog} gives the bound in the statement of the lemma.
\end{proof}

This now implies a uniform bound on $\#\rholog_{\uom}(C(\Q_2))$.

\begin{proposition} \label{P:rhologim}
  Let $C$ be a curve of genus~$g$ over~$\Q_2$. Then
  \[ \#\rholog_{\uom}(C(\Q_2)) \le 288 (g-1)^2 + 129 (g-1) \,. \]
  In particular, $\#\rholog(C(\Q_2)) \le 288 (g-1)^2 + 129 (g-1)$, where $\rholog$ is
  as in~\cite{PoonenStoll2014}.
\end{proposition}

\begin{proof}
  We partition $C(\Q_2)$ into residue disks and annuli according
  to Proposition~\ref{P:DAbounds}. Write $C_D(\Q_2)$ for the union of disks
  and $C_A(\Q_2)$ for the union of annuli. By~\cite{PoonenStoll2014}*{Proposition~5.4}
  (with $p = 2$), we have $\#\rholog_{\uom}(C_D(\Q_2)) \le 5d + 6g - 6$ where
  $d$ is the number of disks. By Proposition~\ref{P:DAbounds},
  $d \le 12(g-1) - 6(t-1)$ and there are at most $2g-3+t$ annuli,
  for some $0 \le t \le g$. This leads to the bound
  \begin{align*}
    \#\rholog_{\uom}(C(\Q_2))
      &\le \#\rholog_{\uom}(C_D(\Q_2)) + \#\rholog_{\uom}(C_A(\Q_2)) \\
      &\le \max_{0 \le t \le g} \bigl\{66(g-1) - 30(t-1) + (2(g-1) + (t-1))(96(g-1) + 31)\bigr\} \\
      &= 288 (g-1)^2 + 129 (g-1)
  \end{align*}
  as claimed. The $\rholog$ map from~\cite{PoonenStoll2014} is
  $\rholog_{\uom}$ for a specific choice of~$\uom$.
\end{proof}

We remark that this bound can be improved somewhat with a bit more work
for hyperelliptic curves~$C$.
For example, one can use the approach of Section~\ref{S:hyp} to get a
partition of~$C(\Q_2)$ into disks and (not necessarily maximal) annuli such that on the
annuli the statement of Corollary~\ref{C:diffann} holds. This gives an improvement
of roughly a factor~$2$, so that the conclusion of Corollary~\ref{C:PSbetter} below
already holds for $g = 17$.
However, it appears that our method will not produce a bound better than
linear in~$g$ for the size of the image of an annulus under~$\rholog$,
and so the final bound for $\#\rholog(C(\Q_2))$ will stay quadratic in~$g$.

We finally obtain a uniformity result for the density of odd degree hyperelliptic
curves with only one rational point in any family defined by congruence conditions,
assuming the genus is sufficiently large.

\begin{corollary} \label{C:PSbetter}
  Let $g \ge 18$ and consider any subfamily~$\calF$ of odd degree hyperelliptic curves
  of genus~$g$ over~$\Q$ defined by finitely many congruence conditions and ordered
  by height as in~\cite{PoonenStoll2014}. Then the lower density of
  curves in~$\calF$ whose only rational point is the point at infinity
  is at least $1 - \bigl(576(g-1)^2 + 258(g-1) + 2) 2^{-g} > 0$.
\end{corollary}

\begin{proof}
  This follows from Proposition~8.13
  of~\cite{PoonenStoll2014} (with $p = 2$), since we know
  from Proposition~\ref{P:rhologim} that (in the notation of~\cite{PoonenStoll2014})
  $\#I \le 288 (g-1)^2 + 129 (g-1)$.
\end{proof}

The lower bound on the density tends to~$1$ quickly as $g \to \infty$,
so we can phrase this result as `most odd degree hyperelliptic curves in
any congruence family have only one rational point.'

%%%%%%%%%%%%%%%%%%%%%%%%%%%%%%%%%%%%%%%%%%%%%%%%%%%%%%%%%%%%%%%%%%%%%%%%%%%


\begin{bibdiv}
\begin{biblist}

\bib{ArtinWinters}{article}{
   author={Artin, M.},
   author={Winters, G.},
   title={Degenerate fibres and stable reduction of curves},
   journal={Topology},
   volume={10},
   date={1971},
   pages={373--383},
   issn={0040-9383},
   review={\MR{0476756 (57 \#16313)}},
}

\bib{Berkovich2007}{book}{
   author={Berkovich, Vladimir G.},
   title={Integration of one-forms on $p$-adic analytic spaces},
   series={Annals of Mathematics Studies},
   volume={162},
   publisher={Princeton University Press},
   place={Princeton, NJ},
   date={2007},
   pages={vi+156},
   isbn={978-0-691-12862-7},
   isbn={0-691-12862-6},
   review={\MR{2263704 (2008a:14035)}},
}

\bib{Bertrand2011preprint}{misc}{
   author={Bertrand, Daniel},
   title={Special points and Poincar\'e bi-extensions},
   date={2011},
   note={Preprint, \texttt{arXiv:1104.5178}, with an appendix by Bas Edixhoven},
}

\bib{Bertrand2013}{article}{
   author={Bertrand, D.},
   title={Unlikely intersections in Poincar\'e biextensions over elliptic
   schemes},
   journal={Notre Dame J. Form. Log.},
   volume={54},
   date={2013},
   number={3-4},
   pages={365--375},
   issn={0029-4527},
   review={\MR{3091662}},
   doi={10.1215/00294527-2143907},
}

\bib{BGR-book}{book}{
   author={Bosch, S.},
   author={G{\"u}ntzer, U.},
   author={Remmert, R.},
   title={Non-Archimedean analysis},
   series={Grundlehren der Mathematischen Wissenschaften [Fundamental
   Principles of Mathematical Sciences]},
   volume={261},
   note={A systematic approach to rigid analytic geometry},
   publisher={Springer-Verlag, Berlin},
   date={1984},
   pages={xii+436},
   isbn={3-540-12546-9},
   review={\MR{746961 (86b:32031)}},
   doi={10.1007/978-3-642-52229-1},
}

\bib{BoschLuetkebohmert-stabredI}{article}{
   author={Bosch, Siegfried},
   author={L{\"u}tkebohmert, Werner},
   title={Stable reduction and uniformization of abelian varieties. I},
   journal={Math. Ann.},
   volume={270},
   date={1985},
   number={3},
   pages={349--379},
   issn={0025-5831},
   review={\MR{774362 (86j:14040a)}},
   doi={10.1007/BF01473432},
}


\bib{BoschLuetkebohmert-stabredII}{article}{
   author={Bosch, Siegfried},
   author={L{\"u}tkebohmert, Werner},
   title={Stable reduction and uniformization of abelian varieties. II},
   journal={Invent. Math.},
   volume={78},
   date={1984},
   number={2},
   pages={257--297},
   issn={0020-9910},
   review={\MR{767194 (86j:14040b)}},
   doi={10.1007/BF01388596},
}

\bib{Buium1993}{article}{
   author={Buium, Alexandru},
   title={Effective bound for the geometric Lang conjecture},
   journal={Duke Math. J.},
   volume={71},
   date={1993},
   number={2},
   pages={475--499},
   issn={0012-7094},
   review={\MR{1233446 (95c:14055)}},
   doi={10.1215/S0012-7094-93-07120-7},
}

\bib{BuiumVoloch}{article}{
   author={Buium, Alexandru},
   author={Voloch, Jos{\'e} Felipe},
   title={Lang's conjecture in characteristic $p$: an explicit bound},
   journal={Compositio Math.},
   volume={103},
   date={1996},
   number={1},
   pages={1--6},
   issn={0010-437X},
   review={\MR{1404995 (98a:14038)}},
}

\bib{CaporasoHarrisMazur1997}{article}{
   author={Caporaso, Lucia},
   author={Harris, Joe},
   author={Mazur, Barry},
   title={Uniformity of rational points},
   journal={J. Amer. Math. Soc.},
   volume={10},
   date={1997},
   number={1},
   pages={1--35},
   issn={0894-0347},
   review={\MR{1325796 (97d:14033)}},
   doi={10.1090/S0894-0347-97-00195-1},
}

\bib{Chabauty1941}{article}{
  author={Chabauty, Claude},
  title={Sur les points rationnels des courbes alg\'ebriques de genre sup\'erieur \`a l'unit\'e},
  language={French},
  journal={C. R. Acad. Sci. Paris},
  volume={212},
  date={1941},
  pages={882\ndash 885},
  review={\MR {0004484 (3,14d)}},
}

\bib{Coleman1985chabauty}{article}{
  author={Coleman, Robert F.},
  title={Effective Chabauty},
  journal={Duke Math. J.},
  volume={52},
  date={1985},
  number={3},
  pages={765\ndash 770},
  issn={0012-7094},
  review={\MR {808103 (87f:11043)}},
}

\bib{ConceicaoUlmerVoloch}{article}{
   author={Concei{\c{c}}{\~a}o, Ricardo},
   author={Ulmer, Douglas},
   author={Voloch, Jos{\'e} Felipe},
   title={Unboundedness of the number of rational points on curves over
   function fields},
   journal={New York J. Math.},
   volume={18},
   date={2012},
   pages={291--293},
   issn={1076-9803},
   review={\MR{2928577}},
}

\bib{Faltings1983}{article}{
  author={Faltings, G.},
  title={Endlichkeitss\"atze f\"ur abelsche Variet\"aten \"uber Zahlk\"orpern},
  language={German},
  journal={Invent. Math.},
  volume={73},
  date={1983},
  number={3},
  pages={349--366},
  issn={0020-9910},
  review={\MR {718935 (85g:11026a)}},
%   translation={ title={Finiteness theorems for abelian varieties over number fields}, booktitle={Arithmetic geometry (Storrs, Conn., 1984)}, pages={9--27}, translator = {Edward Shipz}, publisher={Springer}, place={New York}, date={1986}, },
  note={Erratum in: Invent.\ Math.\ {\bf 75} (1984), 381},
}

\bib{Faltings1994}{article}{
   author={Faltings, Gerd},
   title={The general case of S. Lang's conjecture},
   conference={
      title={Barsotti Symposium in Algebraic Geometry},
      address={Abano Terme},
      date={1991},
   },
   book={
      series={Perspect. Math.},
      volume={15},
      publisher={Academic Press},
      place={San Diego, CA},
   },
   date={1994},
   pages={175--182},
   review={\MR{1307396 (95m:11061)}},
}

\bib{Katz-Zureick-Brown}{article}{
   author={Katz, Eric},
   author={Zureick-Brown, David},
   title={The Chabauty-Coleman bound at a prime of bad reduction and
   Clifford bounds for geometric rank functions},
   journal={Compos. Math.},
   volume={149},
   date={2013},
   number={11},
   pages={1818--1838},
   issn={0010-437X},
   review={\MR{3133294}},
   doi={10.1112/S0010437X13007410},
}

\bib{KRZB}{misc}{
  author={Katz, Eric},
  author={Rabinoff, Joseph},
  author={Zureick-Brown, David},
  title={Uniform bounds for the number of rational points on curves of small Mordell-Weil rank},
  date={2015-04-25},
  note={Preprint, \texttt {arXiv:1504.00694v2 [math.NT]}},
}

\bib{Liu-book}{book}{
   author={Liu, Qing},
   title={Algebraic geometry and arithmetic curves},
   series={Oxford Graduate Texts in Mathematics},
   volume={6},
   note={Translated from the French by Reinie Ern\'e;
   Oxford Science Publications},
   publisher={Oxford University Press},
   place={Oxford},
   date={2002},
   pages={xvi+576},
   isbn={0-19-850284-2},
   review={\MR{1917232 (2003g:14001)}},
}

\bib{Mazur1986}{article}{
   author={Mazur, Barry},
   title={Arithmetic on curves},
   journal={Bull. Amer. Math. Soc. (N.S.)},
   volume={14},
   date={1986},
   number={2},
   pages={207--259},
   issn={0273-0979},
   review={\MR{828821 (88e:11050)}},
   doi={10.1090/S0273-0979-1986-15430-3},
}

\bib{Mazur2000}{article}{
   author={Mazur, Barry},
   title={Abelian varieties and the Mordell-Lang conjecture},
   conference={
      title={Model theory, algebra, and geometry},
   },
   book={
      series={Math. Sci. Res. Inst. Publ.},
      volume={39},
      publisher={Cambridge Univ. Press},
      place={Cambridge},
   },
   date={2000},
   pages={199--227},
   review={\MR{1773708 (2001e:11061)}},
}

\bib{McCallum-Poonen2013}{article}{
   author={McCallum, William},
   author={Poonen, Bjorn},
   title={The method of Chabauty and Coleman},
   language={English, with English and French summaries},
   conference={
      title={Explicit methods in number theory},
   },
   book={
      series={Panor. Synth\`eses},
      volume={36},
      publisher={Soc. Math. France, Paris},
   },
   date={2012},
   pages={99--117},
   review={\MR{3098132}},
}

\bib{Pacelli1997}{article}{
   author={Pacelli, Patricia L.},
   title={Uniform boundedness for rational points},
   journal={Duke Math. J.},
   volume={88},
   date={1997},
   number={1},
   pages={77--102},
   issn={0012-7094},
   review={\MR{1448017 (98b:14020)}},
   doi={10.1215/S0012-7094-97-08803-7},
}

\bib{Pink2005preprint}{misc}{
   author={Pink, Richard},
   title={A Common Generalization of the Conjectures of Andr\'e-Oort, Manin-Mumford, %
          and Mordell-Lang},
   date={2005},
   note={Preprint, \texttt{http://www.math.ethz.ch/$\sim$pink/ftp/AOMMML.pdf}},
}

\bib{PoonenStoll2014}{article}{
   author={Poonen, Bjorn},
   author={Stoll, Michael},
   title={Most odd degree hyperelliptic curves have only one rational point},
   journal={Ann. of Math. (2)},
   volume={180},
   date={2014},
   number={3},
   pages={1137--1166},
   issn={0003-486X},
   review={\MR{3245014}},
   doi={10.4007/annals.2014.180.3.7},
}

\bib{Stoll2006-chabauty}{article}{
  author={Stoll, Michael},
  title={Independence of rational points on twists of a given curve},
  journal={Compos. Math.},
  volume={142},
  date={2006},
  number={5},
  pages={1201--1214},
  issn={0010-437X},
  review={\MR {2264661}},
}

\bib{ZannierBook}{book}{
   author={Zannier, Umberto},
   title={Some problems of unlikely intersections in arithmetic and geometry},
   note={With appendixes by David Masser},
   series={Annals of Mathematics Studies},
   volume={181},
   publisher={Princeton University Press},
   place={Princeton, NJ},
   date={2012},
   pages={xiv+160},
   isbn={978-0-691-15371-1},
}

\end{biblist}
\end{bibdiv}
\end{document}